\UseRawInputEncoding
\documentclass[12pt, reqno]{amsart}
\usepackage[margin=1in]{geometry}
\usepackage{amssymb,latexsym,amsmath,amscd,amsfonts}
\usepackage{latexsym}
\usepackage[mathscr]{eucal}
\usepackage{bm}
\usepackage{mathptmx}
\usepackage{amssymb}
\usepackage{amsthm}
\usepackage{dcolumn}
\usepackage[all]{xy}
\usepackage{enumitem}
\usepackage[utf8]{inputenc}

\def \qed {\hfill \vrule height6pt width 6pt depth 0pt}
\def\textmatrix#1&#2\\#3&#4\\{\bigl({#1 \atop #3}\ {#2 \atop #4}\bigr)}
\def\dispmatrix#1&#2\\#3&#4\\{\left({#1 \atop #3}\ {#2 \atop #4}\right)}
\newcommand{\beg}{\begin{equation}}
	\newcommand{\eeg}{\end{equation}}
\newcommand{\ben}{\begin{eqnarray*}}
	\newcommand{\een}{\end{eqnarray*}}

\newtheorem{thm}{Theorem}[section]
\newtheorem{cor}[thm]{Corollary}
\newtheorem{lem}[thm]{Lemma}

\newtheorem{prop}[thm]{Proposition}
\numberwithin{equation}{section} \theoremstyle{definition}
\newtheorem{defn}[thm]{Definition}
\newtheorem{rem}[thm]{Remark}

\newtheorem{eg}[thm]{Example}

\newcommand{\HS}{\mathcal H}
\newcommand{\C}{\mathbb{C}}

\def\textmatrix#1&#2\\#3&#4\\{\bigl({#1 \atop #3}\ {#2 \atop #4}\bigr)}
\def\dispmatrix#1&#2\\#3&#4\\{\left({#1 \atop #3}\ {#2 \atop #4}\right)}

\title[Positive definite functions on a group]{Remarks on positive definite functions on a group}
\author[Jana, Pal and Tomar]{SWAPAN JANA, SOURAV PAL AND NITIN TOMAR}

\address[Swapan Jana]{Mathematics Department, Indian Institute of Technology Bombay,
	Powai, Mumbai - 400076, India.} \email{swapan.jana@iitb.ac.in} 

\address[Sourav Pal]{Mathematics Department, Indian Institute of Technology Bombay,
	Powai, Mumbai - 400076, India.} \email{sourav@math.iitb.ac.in}

\address[Nitin Tomar]{Mathematics Department, Indian Institute of Technology Bombay, Powai, Mumbai-400076, India.} \email{tnitin@math.iitb.ac.in}		



\thanks{The first named author is supported by the Prime Minister's Research Fellowship (PMRF ID 1302045), Government of India. The second named author is supported by the Seed Grant of IIT Bombay, the CDPA and the MATRICS Award (Award No. MTR/2019/001010) of Science and Engineering Research Board (SERB), India. The third named author is supported by the Prime Minister's Research Fellowship (PMRF ID 1300140), Government of India.}	
\begin{document}
	\maketitle

	\begin{abstract}
		We study the operator-valued positive definite functions on a group using positive block matrices. We give an alternative proof to Brehmer positivity for doubly commuting contractions. We classify all commuting unitary representations over a finite group. We show by examples that the power of a positive-definite function may not be positive definite and also the power of a unitary representation may not be a unitary representation. We also characterize all unitary representations whose powers are also unitary representations.
		\end{abstract} 
\section{Introduction}	
	
\vspace{0.3cm}		
	
	\noindent  Throughout the paper, all operators are bounded linear maps acting on complex Hilbert spaces. For a Hilbert space $\HS$, we denote by $\mathcal B(\HS)$ the algebra of operators on $\HS$ with identity $I_\HS$. A contraction is an operator with norm atmost one. Given a contraction $T$, $D_T$ denotes the unique positive square root of $I_\HS-T^*T$. Also, $\mathbb{D}$ and $\mathbb{Z}$ denote the unit disk in the complex plane $\mathbb{C}$ and the group of integers respectively. 
	
	\medskip

We begin with a celebrated theorem due to Sz.-Nagy which states that every contraction can be realized as a part of a unitary operator.
	
	\begin{thm}[Sz.-Nagy, \cite{NagyFoias}]\label{Nagy}
	
		If $T$ is a contraction acting on a Hilbert space $\mathcal{H}$, then there exist a Hilbert space $\mathcal{K} \supseteq \mathcal{H}$ and a unitary $U$ on $\mathcal{K}$ such that 
		\[
		T^k=P_\mathcal{H}U^k|_{\mathcal{H}}
		\]
for all $k \in \mathbb{N}\cup \{0\}$. Moreover, $\mathcal{K}$ can be chosen to be minimal in the sense that $\mathcal{K}$ is the smallest closed reducing subspace for $U$ that contains $\HS$.
\end{thm}

There are several proofs in the literature of this fundamental result, e.g. \cite{Nagy+, NagyFoias, Schaffer}. One particular proof (see Section 8.1 in \cite{NagyFoias}) capitalizes the theorem of Naimark \cite{NeumarkI} on operator-valued positive definite functions on a group. Since then positive definite functions on a group have been well-studied in the context of dilation, e.g. see \cite{Foias and Frazho, NagyFoias} and the references therein. In this article, we further investigate behaviour of positive definite operator-valued functions on a group. 

\begin{defn}

	Let $G$ be a group with identity $e$. A map $T: G \to \mathcal{B}(\HS)$ is said to be

	\begin{enumerate}

		\item \textit{(strictly) positive definite} if $T\left(s^{-1}\right)=T(s)^* \text { for every } s \in G$ and the sum
				\begin{equation*}
			\sum_{s \in G} \sum_{t \in G} \langle T(t^{-1} s) h(s), h(t)\rangle, 
		\end{equation*}
		is (strictly) positive for every $h \in c_{00}(G, \mathcal{H}).$ Here the set $c_{00}(G, \mathcal{H})$ denotes the collection of functions from $G$ to $\mathcal{H}$ which takes non-zero values on a finite subset of $G$ only.
		\vspace{0.2cm}
		\item a \textit{unitary representation} if each $T(s)$ is a unitary operator on $\mathcal{H}$ such that $T(e)=I_\HS$ and $T(s)T(t)=T(st)$ for $s, t \in G$.

	\end{enumerate} 

\end{defn}


Given an operator $T \in \mathcal{B}({\HS})$, Sz.-Nagy \cite{Nagy+} realized that the map on $\mathbb{Z}$ given by

\begin{equation}\label{T + on Z}
m \mapsto \left\{
\begin{array}{ll}
	T^m & m \geq 1\\
	I_{\HS} & m = 0\\
	T^{*|m|} & m <0\\
\end{array} 
\right.
\end{equation}
is positive definite if and only if $T$ is a contraction. Consequently, as a special case of the following theorem due to Naimark we obtain a unitary dilation of a contraction $T$ (as given in Theorem \ref{Nagy}).
	\begin{thm}[Naimark, \cite{NeumarkI}]\label{NeumarkI}
		
	For every positive definite function $T(s)$ on $G$, whose values are operators on a Hilbert space $\mathcal{H},$ with $T(e)=I_\mathcal{H},$ there is a
	unitary representation $U(s)$ on $G$ on a space $\mathcal{K}$ containing $\mathcal{H}$ as subspace, such that
	\begin{equation*}\label{dilation}
		T(s)=P_\mathcal{H}U(s)|_\mathcal{H} \quad(s \in G)    
	\end{equation*}
	and
	\begin{equation*}\label{minimality}
		\mathcal{K}=\bigvee_{s \in G}U(s)V\mathcal{H} \quad(\mbox{minimality condition}).    
	\end{equation*}
	This unitary representation of $G$ is determined by the function $T(s)$ upto isomorphism. Conversely, given  a
	unitary representation $U(s)$ on $G$ on a space $\mathcal{K}$ and a subspace $\mathcal{H}$ of $\mathcal{K},$ the map $T: G \to \mathcal{B}(\mathcal{H})$ defined by 
	\[
	T(s)=P_\mathcal{H}U(s)|_\mathcal{H} \quad(s \in G).
	\]
	is a positive definite function with $T(e)=I_\mathcal{H}.$
	
\end{thm} 

\medskip

 One can associate certain block matrices to an operator-valued function $T$ on a group and show that $T$ is positive definite if and only if the block matrices are positive, e.g. see \cite{Foias and Frazho}. We discuss such correspondences via examples in Section \ref{Sec_+ functions, matrices}. We find characterizations of positive definite functions on the groups of orders upto 4. In Section \ref{+ on Z^n}, we study positive definite functions on $\mathbb{Z}$ and $\mathbb{Z}\oplus \mathbb{Z}$. As an application, we provide an alternative proof ot the fact that for $T \in \mathcal{B}(\HS)$, the map given in (\ref{T + on Z}) is positive definite if and only if $\|T\| \leq 1$. We further prove a widely known extension of this result (see Section 9 in \cite{NagyFoias}) to the two-variables setting, namely the Brehmer positivity, in Theorem \ref{+ on Z+Z} using the block matrix technique.

\medskip

About positive definite functions and unitary representations, we move in two directions. First, in Section \ref{structure thm} we characterize a unitary representation on a finite group whose image consists of the commuting unitaries. We find such unitary representations explicitly for the symmetric and dihedral groups. Second, we address the following question: if $T$ is a positive definite function on a group $G$, then does $T(s)^n=P_\HS U(s)^n|_\HS$ hold for every  $s \in G$ and $n \in \mathbb{N}$ ? To answer this, we define $T_n: G \to \mathcal{B}(\mathcal{H})$ as $s \mapsto T(s)^n$ for $n \in \mathbb{N}$. We also explore if $T_n$ is positive definite when $T$ is so and what happens when $T$ is a unitary representation. In Section \ref{power dilation}, we show that these two questions have negative answers. However, it turns out that if $\{T(s)\}$ consists of commuting matrices or commuting normal operators, then $T_n$ is positive definite for every $n \in \mathbb{N}$. We conclude the article by discussing if $T_n(s)=P_\HS U_n(s)|_\HS$ whenever $T_n$ and $U_n$ define a positive definite function and unitary representation respectively. We show that even with such a strong hypothesis, the result fails to be true.

\vspace{0.2cm}

	\section{Some elementary results}\label{Sec_Neumark}
	
	\vspace{0.2cm}

 \noindent Recall that for a group $G$ with identity $e$, Naimark's Theorem (Theorem \ref{NeumarkI}) states that every $\mathcal{B}(\HS)$-valued positive definite function $T$ on $G$ with $T(e)=I_\HS$ gives rise to a unitary representation $U$ on $G$ such that $T(s)=P_\HS U(s)|_\HS$ for every $s$ in $G$. In the literature, there is a generalization of Naimark's Theorem by omitting the hypothesis that $T(e)=I_\HS$. Indeed, one can still obtain a unitary representation $U$ but $T(s) \ne P_\HS U(s)|_\HS$ for some $s$ in $G$. This gives a clear idea about the importance of the assumption $T(e)=I_\HS$ in Naimark's Theorem. En route, we obtain some useful results as a direct consequence of the generalization.  We start with the following lemma.\\


\begin{lem}\label{contractionsI}
	
	 Let $G$ be a group with identity $e$ and let $\HS$ be a Hilbert space. For a $\mathcal{B}(\HS)$-valued positive definite function $T$ on $G$ if $T(e)=I_\HS$, then $\|T(s)\| \leq 1$ for every $s \in G$.

\end{lem}

\begin{proof} 
	
		 For $x \in \mathcal{H}$ and $a \in G$, let $T(a)x = -x_1$. Let us define $h: G \rightarrow \HS$ by
		\[
		h(s)=\left\{
		\begin{array}{ll}
			x_1 & s=e\\
			x &  s= a\\
			0 & s \ne e,a \ .
		\end{array} 
		\right.
		\]
		Then, evidently $h \in c_{00}(G, \mathcal{H})$. For this particular $h$, we have 
		\begin{equation*}
			\begin{split}
				0 \leq \sum_{s \in G} \sum_{t \in G} \langle T(t^{-1} s) h(s), h(t)\rangle &= \langle T(e)x_1, x_1\rangle  + \langle T(a)x, x_1 \rangle+ \langle T(a^{-1})x_1, x \rangle+\langle T(e)x, x\rangle\\
				&= \langle x_1, x_1\rangle  + \langle T(a)x, x_1 \rangle+ \langle T(a)^{*}x_1, x \rangle+\langle x, x\rangle\\
				&= \|T(a)x\|^2  - \langle T(a)x, T(a)x \rangle- \langle T(a)x, T(a)x \rangle+\| x\|^2\\
				&=\|x\|^2-\|T(a)x\|^2.\\
			\end{split}
		\end{equation*}
		Thus, $T(a)$ is a contraction for every $a \in G.$ 
	\end{proof}
 It is naturally asked if the converse holds i.e. for a $\mathcal B(\HS)$-valued positive definite function $T$ on a group $G$, if $\{T(a): a \in G\}$ is a family of contractions, then is $T(e)=I_{\HS}$ ? The following lemma shows that it is not true in general and we characterize all such positive definite functions.
	\begin{lem}\label{contractionsII}
		
		Let $T$ be a $\mathcal B(\HS)$-valued positive definite function acting on a group $G$. Then the range of $T$ consists of contractions if and only if $T(e)$ is a contraction.
	
\end{lem}
	\begin{proof}
		The forward part is trivial. We assume that $T(e)$ is a contraction. For any $x \in \HS$ and $a \in G$, if $T(a)x= -x_1$ then It follows from the proof of Lemma \ref{contractionsI} that 
		\begin{equation}\label{>0}
			0 \leq \langle T(e)x_1, x_1\rangle  + \langle T(a)x, x_1 \rangle+ \langle T(a^{-1})x_1, x \rangle+\langle T(e)x, x\rangle.
		\end{equation}
		Substituting $x=0$ in (\ref{>0}) we have that $T(e)$ is a positive operator. Also, if we write $x_1=-T(a)x$ in (\ref{>0}), then by an application of Cauchy-Schwarz inequality, we have
		\begin{equation*}
			\begin{split}
				0 & \leq \|T(e)\|\cdot \|x_1\|^2+ \langle T(a)x, x_1 \rangle+ \langle T(a)^*x_1, x \rangle+\|T(e)\|\cdot \|x\|^2\\ 
				& =\|T(e)\|\cdot \|T(a)x\|^2 -\langle T(a)x, T(a)x \rangle- \langle T(a)x, T(a)x \rangle+\|T(e)\|\cdot \|x\|^2\\
				& \leq  \|T(a)x\|^2 -\|T(a)x\|^2- \|T(a)x\|^2 + \|x\|^2 \quad(\because \|T(e)\| \leq 1 )\\
				&=\|x\|^2-\|T(a)x\|^2.\\
			\end{split}
		\end{equation*}
		Thus $T(a)$ is a contraction for every $a \in G$ and the proof is complete.
	\end{proof}
 Now, we present a generalized Naimark's Theorem where the hypothesis that $T(e)=I_{\HS}$ is omitted. Interestingly, we still obtain a unitary representation $U$ on the same group. However, we should not expect that compression of $U(s)$ gives back $T(s)$ for every $s$ as $T(s)$ may not always be a contraction.      
	\begin{thm}[\cite{Paulsen}, Theorem 4.8]\label{NeumarkII}
		
	Let $G$ be a group and let $\HS$ be a Hilbert space. For every $\mathcal{B}(\mathcal{H})$-valued positive definite function $T(s)$ on $G$, there exist a Hilbert space $\mathcal{K}$, an operator $V: \mathcal{H} \to \mathcal{K}$ with $V^*V=T(e)$ and a
		unitary representation $U(s)$ on $G$ on a space $\mathcal{K}$ such that
		\begin{equation}\label{dilation}
			T(s)=V^*U(s)V \quad(s \in G)    
		\end{equation}
		and
		\begin{equation}\label{minimality}
			\mathcal{K}=\bigvee_{s \in G}U(s)V\mathcal{H} \quad(\mbox{minimality condition}).    
		\end{equation}
		Also, this unitary representation $U(s)$ is unique upto an isomorphism. Conversely, given an operator $V: \mathcal{H} \to \mathcal{K}$ and a
		$\mathcal B(\mathcal K)$-valued unitary representation $U(s)$ on $G$, the map $T: G \to \mathcal{B}(\mathcal{H})$ defined by 
		\[
		T(s)=V^*U(s)V \quad(s \in G)
		\]
		is a positive definite function.
	\end{thm} 
	
Evidently, it follows that $V$ is an isometry if and only if $T(e)=I_\HS.$ By an application of Theorem \ref{NeumarkII}, one can have a stronger version of Lemma \ref{contractionsII} as given below.
	\begin{prop}\label{contractionsIII}
		Let $T$ be a positive definite function on a group $G$. Then $\|T(s)\| \leq \|T(e)\|$ for every $s\in G$.
	\end{prop}
	\begin{proof}
		It follows from Theorem \ref{NeumarkII} that there is an operator $V: \mathcal{H} \to \mathcal{K}$ with $\|V\|^2=\|T(e)\|$ and a
		unitary representation $U$ on $G$ such that $ T(s)=V^*U(s)V$ for all $s \in G$. Therefore, for any $s\in G$ we have
		\begin{equation} \label{eqn:Sec3-01}
		\|T(s)\| =\|V^*U(s)V\| \leq \|V^*\|\cdot \|U(s)\| \cdot \|V\| =\|V\|^2=\|T(e)\|.
		\end{equation}
		  
	\end{proof}

Needless to mention that every $T(s)$ is a contraction if and only if $T(e)$ is a contraction. Also, if $T(e)=0$, then $T$ is indeed the zero function.

\vspace{0.2cm}
	\section{Positive definite functions on small groups and block matrices}\label{Sec_+ functions, matrices}
	
	\vspace{0.2cm}
\noindent	 There is a natural way to associate a family of block matrices to an operator-valued function $T$ acting on a group $G$ such that the positivity of all blocks is equivalent to the positive definiteness of $T$. An interested reader is referred to \cite{Foias and Frazho} for further details. Here we find an explicit description of all positive definite functions on some well-known groups. We begin with a basic lemma.	
	\begin{lem}\label{V^*TV} Let $T: G \to \mathcal{B}(\mathcal{H})$ be a positive definite function and let $V\in\mathcal{B}(\mathcal{H})$ be arbitrary. Then the map $T_{V} : G \to \mathcal{B}(\mathcal{H})$ defined by
		$T_{V}(s)=V^{*}T(s)V$ is positive definite.
		\begin{proof} First note that
			\[
			T_V(s^{-1})=V^*T(s^{-1})V=V^*T(s)^{*}V=(V^*T(s)V)^{*}=T_V(s)^{*},
			\]		
			for every $s \in G.$ For every $h\in c_{00}(G,\mathcal{H}),$ we have that
			
			\[
			\sum_{s,t\in G}\langle T_V(s^{-1}t) h(t),h(s)\rangle=\sum_{s,t\in
				G}\langle V^{*} T(s^{-1}t) V h(t),h(s)\rangle=\sum_{s,t\in G}\langle
			T(s^{-1}t)V h(t), Vh(s)\rangle \geq 0.
			\]
The last inequality holds because the map $s \mapsto V(h(s))$ is in $c_{00}(G, \mathcal{H})$ and $T$ is positive definite.
		\end{proof}
	\end{lem}
	
	Next, we write down the block matrices associated with a positive definite
	function defined on a group $G.$ Let  $h\in c_{00}(G,\mathcal{H})$ with support $ \{s_1,\dotsc, s_m\} \subseteq G$. It is easy to see that
		\begin{equation*}
		\begin{split}
			\sum_{s,t\in G}\langle T(s^{-1}t)h(t),h(s)\rangle &=\sum_{i=1}^{m}\sum_{j=1}^{m}\langle T(s_j^{-1}s_{i})h(s_{i}),h(s_j)\rangle\\
			&=\sum_{j=1}^{m}\bigg\langle \sum_{i=1}^{m} T(s_j^{-1}s_{i})h(s_{i}),h(s_j)\bigg\rangle\\
			&=\Biggl\langle 
				\begin{bmatrix}
				T(s_{1}^{-1}s_1) & T(s_{1}^{-1}s_2) & \cdots & T(s_{1}^{-1}s_m) \\
				T(s_{2}^{-1}s_1) & T(s_{2}^{-1}s_2) & \cdots & T(s_{2}^{-1}s_m) \\
				\vdots          & \vdots          & \cdots & \vdots\\
				T(s_{m}^{-1}s_1) & T(s_{m}^{-1}s_2) & \cdots & T(s_{m}^{-1}s_m)
			\end{bmatrix}
			\begin{bmatrix}
				h(s_1)\\
				h(s_2)\\
				\vdots\\
				h(s_m)\\
			\end{bmatrix},\begin{bmatrix}
				h(s_1)\\
				h(s_2)\\
				\vdots\\
				h(s_m)\\
			\end{bmatrix}\Biggr\rangle\\
			& = \langle \Delta_T(s_1, \dotsc, s_m) x, x\rangle_{\mathcal{H}^m} \ ,
		\end{split}
	\end{equation*}
	where, $x=\left(h(s_1),\dotsc,h(s_m)\right)$ is in $\mathcal{H}^m$. The block matrix $\Delta_T$ is given by
	
	\[ 
			\Delta_T(s_1, \dotsc, s_m)=\begin{bmatrix}
		T(s_{1}^{-1}s_1) & T(s_{1}^{-1}s_2) & \cdots & T(s_{1}^{-1}s_m) \\
		T(s_{2}^{-1}s_1) & T(s_{2}^{-1}s_2) & \cdots & T(s_{2}^{-1}s_m) \\
		\vdots          & \vdots          & \cdots & \vdots\\
		T(s_{m}^{-1}s_1) & T(s_{m}^{-1}s_2) & \cdots & T(s_{m}^{-1}s_m)
	\end{bmatrix}=\begin{bmatrix}
	T(s_i^{-1}s_j)
\end{bmatrix}_{1 \leq i,j \leq m}.
	\]
	Consequently, we have the following result.
	\begin{prop}[\cite{Foias and Frazho}, Chapter XV]\label{+ arbitrary}
		Let $G$ be a group and let $\HS$ be a Hilbert space. Then $T: G \to \mathcal{B}(\mathcal{H})$ is a positive definite function if and only if  $\Delta_T(s_1, \dotsc, s_m)$ is  positive in $\mathcal{B}(\mathcal{H}^{m})$ for every finite set $\{s_1, \dotsc, s_m\}$ in $G$.
	\end{prop}

 It is evident from the above discussion that the positive definiteness of a function $T$ on a finite group $G$ depends only on a single block matrix of operators, namely $\Delta_T(G)$. So, we have the following proposition.

	\begin{prop}\label{+ finite}
		Let $G=\{s_i: 1 \leq i \leq m\}$ be a finite group. Then $T: G \to \mathcal{B}(\mathcal{H})$ is a positive definite function if and only if the operator $\Delta_T(G)$ is  positive in $\mathcal{B}(\mathcal{H}^{m}).$
	\end{prop}
	\begin{rem}
		Note that the positivity of the matrix $\Delta_T(G)$ is independent of the arrangements of the elements of $G$, when $G$ is finite. For a postive definite function $T$ on a finite group $G$ if $\{s_1, \dotsc, s_m\}$ and $\{x_1, \dotsc, x_m\}$ are two arrangements of the elements of $G$ with associated blocks $\Delta_T(G;s)$ and $\Delta_T(G;x)$ respectively, then $U^*\Delta_T(G; x)U=\Delta_T(G; s)$, where $U$ is the permutation matrix corresponding to the permutation $x_i = s_{\sigma{i}}$, $i=1, \dots , m$. Then it follows from Lemma \ref{V^*TV} that $\Delta_T(G;s)$ is positive if and only if $\Delta_T(G;x)$ is positive.
	\end{rem}

Let us find an explicit form of $\Delta_T(G)$ for the following groups.
	\begin{eg}[Positive definite function on finite cyclic groups]
		Let $T: \mathbb{Z}_n \to \mathcal{B}(\mathcal{H})$ be a positive
		definite function. Let us denote by $T(k)=T_k$ for $k=0,1, \dotsc, n-1$. We consider two different cases depending on whether $n$ is odd or even. 
		\begin{enumerate}\label{+ Z_n}
			\item Let $n$ be odd and let $m=(n-1)\slash 2$. Then the associated block matrix $\Delta_T$ is given by
			\[ 
			\Delta_T(\mathbb{Z}_n)=\begin{bmatrix}
				
				T_0 & T_1 & T_2 & \cdots & T_m & T_{m}^{*} & T_{m-1}^{*} & \cdots & T_{1}^{*}\\
				
				T_{1}^{*} & T_0 & T_1 & \cdots & T_{m-1} & T_{m} & T_{m}^{*} &\cdots & T_{2}^{*}\\
				
				T_{2}^{*} & T_{1}^{*} & T_0 & \cdots & T_{m-2} & T_{m-1} & T_{m} &
				\cdots & T_{3}^{*}\\
				
				\vdots & \vdots & \vdots & \ddots & \vdots & \vdots & \vdots &\ddots &\vdots\\
				
				T_1 & T_2 & T_3 &\cdots & T_{m}^{*} & T_{m-1}^{*} & T_{m-2}^{*} &\cdots & T_0
				
			\end{bmatrix}_{n\times n}.
			\]
			
			\item Let $n$ be even and let $m=n\slash 2$. Here, the associated block matrix $\Delta_T$ is given by
			\[ 
			\Delta_T(\mathbb{Z}_n)=\begin{bmatrix}
				
				T_0 & T_1 & T_2 & \cdots & T_{m-1} & T_{m} & T_{m-1}^{*} & \cdots & T_{1}^{*}\\
				
				T_{1}^{*} & T_0 & T_1 & \cdots & T_{m-2} & T_{m-1} & T_{m} &\cdots & T_{2}^{*}\\
				
				T_{2}^{*} & T_{1}^{*} & T_0 & \cdots & T_{m-3} & T_{m-2} & T_{m-1} &
				\cdots & T_{3}^{*}\\
				
				\vdots & \vdots & \vdots & \ddots & \vdots & \vdots & \vdots &\ddots &\vdots\\
				
				T_1 & T_2 & T_3 &\cdots & T_{m}^{*} & T_{m-1}^{*} & T_{m-2}^{*} &\cdots & T_0
				
			\end{bmatrix}_{n\times n}.
			\]
		\end{enumerate}	
		In either case, it follows from Proposition \ref{+ finite} that $T$ is positive
		definite if and only if $\Delta_T(\mathbb{Z}_n) \geq 0$. \qed
		
	\end{eg}
	\begin{eg}[Positive definite function on dihedral groups] Let $D_n$ be the dihedral group of order $2n$ which has the following representation.
		\[
		D_n=\langle r,s \ | \ r^n=s^2=(sr)^2=e \rangle.
		\]
		Let $T: D_n \to \mathcal{B}(\mathcal{H})$ be positive definite. Set $T(r^m)=T_m$ and $T(sr^m)=V_m$ for $ m=0,1, \dotsc, n-1$. A few steps of routine calculations yield 
		\[
		\Delta_T(D_n)=\begin{bmatrix}
			\Delta_T(\mathbb{Z}_n) & \Delta_V\\
			\Delta_V & \Delta_T(\mathbb{Z}_n)\\
		\end{bmatrix},
		\]
		where $\Delta_T(\mathbb{Z}_n)$ is the block matrix given in Example \ref{+ Z_n} and $\Delta_V$ is given by
		\[
		\Delta_V=\begin{bmatrix}
			V_0 & V_1 & \cdots & V_{n-2} & V_{n-1}\\
			V_1 & V_2 & \cdots & V_{n-1} & V_{0}\\
			
			\vdots & \vdots &  \cdots & \vdots & \vdots\\
			V_{n-1} & V_0 & \cdots & V_{n-3} & V_{n-2}\\
		\end{bmatrix}_{n\times n}.
		\]
		Since $T$ is positive definite, we have 
		\[
		V_m=T(sr^m)=T((sr^m)^{-1})=T(sr^m)^*=V_m^* \, \quad \text{ for } \quad 0 \leq m \leq n-1.
		\]
	\qed	
		
		\end{eg}
Now we prove that the positive definiteness of an operator-valued function $T$  is preserved under group isomorphism. This gives us a way to consider only positive definite functions on any isomorphic copy of a group. 
	\begin{lem}
		Let $T: G \to \mathcal{B}(\mathcal{H})$ be a positive definite function and let $\phi:G_0 \to G $ be an isomorphism of groups. Then the map 
$
		T_\phi: G_0 \to \mathcal{B}(\mathcal{H})$ given by  $T_\phi(g_0)=T(\phi(g_0))$
		is positive definite.
	\end{lem}
	\begin{proof}
Since $\phi$ is a group homomorphism, $T_\phi(g_0^{-1})=T_\phi(g_0)^*$.
For every $h\in c_{00}(G_0,\mathcal{H})$,we have
\begin{equation*}
	\begin{split}
		\sum_{s,t\in G_0}\langle T_\phi(s^{-1}t) h(t),h(s)\rangle
		&=\sum_{s,t\in
			G_0}\langle  T(\phi(s^{-1}t)) h(t),h(s)\rangle\\
		&=\sum_{s,t\in
			G_0}\langle  T\left(\phi(s)^{-1}\phi(t)\right) h(t),h(s)\rangle\\
		&=\sum_{\phi(s), \phi(t)\in
			G}\langle  T\left(\phi(s)^{-1}\phi(t)\right) (h\circ\phi^{-1})\phi(t),(h\circ \phi^{-1})\phi(s)\rangle\\
			& \geq 0.\\
	\end{split}
\end{equation*}
The last inequality follows from the facts that for $h \in c_{00}(G_0, \mathcal{H)}$ the map $s \mapsto (h\circ\phi^{-1})(s)$ is in $c_{00}(G, \mathcal{H})$ and $T$ is a positive definite function.
	\end{proof}
 Next, we recollect some results from the literature that are crucial for this section. 	
\begin{thm}[\cite{Foias and Frazho}, Chapter XVI, Theorem 1.1]\label{2*2I}
	Consider the block matrix $T$ defined by 
	\[
	\begin{bmatrix}
		A & B \\
		B^* & C
	\end{bmatrix} \quad \mbox{on } \ \ \mathcal{H}_1 \oplus \mathcal{H}_2 \ ,
	\]
where $A, B$ and $C$ are operators acting between the appropriate spaces. The operator $T$ is (strictly) positive if and only if $A$	and $C$ are both (strictly) positive and there is a (strict) contraction $\Gamma$ mapping $\overline{C\mathcal{H}_2}$ into $\overline{A\mathcal{H}_1}$ satisfying 
\[
B=A^{1\slash2}\ \Gamma \ C^{1\slash 2}.
\]
\end{thm}	
 Also, we have the following result.
\begin{prop}[\cite{Kian}, Proposition 5.18]\label{2*2II}
	Let $A, B \in \mathcal{B}(\mathcal{H})$ be such that $B$ is self-adjoint and $A$ is positive. Then 
		\[
	\begin{bmatrix}
		A & B \\
		B & A
	\end{bmatrix} \geq 0 \quad \mbox{if and only if} \quad \pm B \leq A.
	\]
\end{prop}
 
Now we  discuss positive definite functions on some small groups. 
	\subsection{Positive definite functions on group of order 2}
	
	Every group of order $2$ is isomorphic to $\mathbb{Z}_{2}.$ Therefore, we characterize positive definite functions on $\mathbb{Z}_2.$ The following results follow directly from Example \ref{+ Z_n}, Theorem \ref{2*2I} and  Proposition \ref{2*2II}.
	\begin{prop}\label{+ on Z2}
		Let $T\colon\mathbb{Z}_{2}\to\mathcal{B}(\mathcal{H})$ be a function such that $T(0)$ and $T(1)$ are self-adjoint operators. Then the following are equivalent.
		\begin{enumerate}
			\item $T$ is a positive definite function;
			\item The block matrix \[ T=\begin{bmatrix}
				T(0) & T(1) \\
				T(1) & T(0)  \\
			\end{bmatrix}_{2\times 2}\] is a bounded positive operator on $\mathcal{B}(\mathcal{H}^{2});$ 
		\item There is a contraction $\Gamma$ mapping $\overline{T(0)(\mathcal{H})}$ into $\overline{T(1)(\mathcal{H})}$ such that 
		\[
		T(1)=T(0)^{1\slash 2 } \ \Gamma \ T(0)^{1\slash 2};
		\]
		
		\item $T(0)$ is positive and $\pm T(1)\leq T(0).$
		\end{enumerate}		
	\end{prop}
	
	\begin{cor}\label{+ on Z_2II}

		Let $T\colon\mathbb{Z}_{2}\to\mathcal{B}(\mathcal{H})$ be such that $T(0)=I_{\mathcal{H}}$ and $T(1)=T(1)^*$. Then $T$ is positive definite if and only if $\|T(1)\| \leq 1$.
	
\end{cor}	

		\begin{proof}
It follows from Proposition \ref{+ on Z2} that $T$ is positive definite if and only if there is a contraction $\Gamma$ such that $T(1)=T(0)^{1\slash 2 }\ \Gamma \ T(0)^{1\slash 2}=\Gamma$ which is possible if and only if $T(1)$ is a contraction since $T(0)=I_\HS.$ 
		\end{proof}

	\begin{cor}\label{2 by 2}
		Let $T\colon\mathbb{Z}_{2}\to\mathcal{B}(\mathcal{H})$ be such that $T(0)$ and $T(1)$ are self-adjoint. Then $T$ is strictly positive definite
		if and only if $T(0)$ is strictly positive and $
		  T(0)^{-1\slash2} T(1) T(0)^{-1\slash2}
		  $
		   is a strict contraction.
		\begin{proof}
			The desired conclusion follows from Theorem \ref{2*2I}.
		\end{proof}
	\end{cor}
	
	\subsection{Positive definite functions on a group of order 3}
	
	We wish to provide necessary and sufficient conditions for an operator-valued map acting on a group of order $3$ to be positive definite. Since every group of order $3$ is isomorphic to $\mathbb{Z}_{3}$, it suffices to study the positive definite functions on $\mathbb{Z}_3$. We begin with the following result.

\begin{thm}[\cite{Foias and Frazho}, Chapter XVI, Theorem 3.1] \label{+ on 3*3}
	Consider the block matrix $T$ defined by 
	\[
	\begin{bmatrix}
		A & B & R\\
		B^* & C & B'\\
		R^* & B'^* & D
	\end{bmatrix} \quad \mbox{on} \ \mathcal{H}_1 \oplus \mathcal{H}_2 \oplus \mathcal{H}_3
	\]
	where $A, B, B', C, D$ and $R$ are operators acting between the appropriate spaces. The operator $T$ is (strictly) positive if and only if the operators 
\[
	\begin{bmatrix}
		A & B \\
		B^* & C
	\end{bmatrix} \quad \mbox{on} \ \mathcal{H}_1 \oplus \mathcal{H}_2, \quad 
\begin{bmatrix}
C & B' \\
B'^* & D
\end{bmatrix} \quad \mbox{on} \ \mathcal{H}_2 \oplus \mathcal{H}_3
	\]
are (strictly) positive and
\[
R=A^{1\slash 2}D_{\Gamma^*}\Gamma_RD_{\Gamma'}D^{1\slash 2}+A^{1\slash 2}\Gamma \Gamma' D^{1\slash2},
\]	
where $\Gamma_R$ is a contraction mapping $\mathfrak{D}_{\Gamma'}$ into $\mathfrak{D}_{\Gamma^*}, \Gamma$ is a contraction mapping $\overline{C\mathcal{H}_2}$ into $\overline{A\mathcal{H}_1}$ satisfying 
	\[
	B=A^{1\slash2}\Gamma C^{1\slash 2}.
	\]
and $\Gamma'$ is a contraction mapping $\overline{D\mathcal{H}_3}$ into $\overline{C\mathcal{H}_2}$ satisfying 
	\[
	B'=C^{1\slash2}\Gamma' D^{1\slash 2}.
	\]
\end{thm}	

 The following result is a direct consequence of Theorem \ref{+ on 3*3} and Example \ref{+  Z_n}.
	\begin{prop}\label{+ on Z_3} 
		For a $\mathcal{B}(\mathcal{H})$-valued map $T$ acting on $\mathbb{Z}_3$, the following are equivalent.
		
		\begin{enumerate}
			\item $T$ is a positive definite function;
			\item  The block matrix 
			\[ 
			\Delta_T=\begin{bmatrix}
				T(0) & T(1) & T(1)^{*} \\
				T(1)^{*} & T(0) & T(1)^{*} \\
				T(1) & T(1)^{*} & T(0)\\
			\end{bmatrix}_{3\times 3}
		\] is a bounded positive operator on $\mathcal{H}^{3};$  
		
		\item $T(0)$ is positive and there are contractions $\Gamma_{0}\colon\overline{T(0)\mathcal{H}}\to\overline{T(0)\mathcal{H}}$ and $\Gamma_{1}\colon\mathfrak{D}_{\Gamma_{0}}\to\mathfrak{D}_{\Gamma_{0}^{*}}$ such that 
		\[
		T(1)=T(0)^{1\slash2} \Gamma_{0} T(0)^{1\slash2} 
		\]
		and
		\[
		 T(1)^{*}=T(0)^{1\slash2}D_{\Gamma_{0}^{*}}\Gamma_{1}D_{\Gamma_{0}}T(0)^{1\slash2} + T(0)^{1\slash2} \Gamma_{0}^{2} T(0)^{1\slash2}
		\]
			\end{enumerate} 
\end{prop}
	
		\begin{proof}
	In Example \ref{+ Z_n}, we have proved that $T$ is a positive definite function if and only if the block matrix $\Delta_T$ is positive. Substituting the entries from $\Delta_T$ in Theorem \ref{+ on 3*3}, we obtain that $\Delta_T$  is positive if and only if condition $(3)$ in the statement holds.
		\end{proof}

	 For the next corollary, we shall use the notations as in Proposition \ref{+ on Z_3}.
	
	\begin{cor}
		Let $T\colon\mathbb{Z}_{3}\to\mathcal{B}(\mathcal{H})$ be such that $T(0)=I_{\mathcal{H}}$ and $T(1)=T(2)^{*}=T_1$ is a contraction. Then $T$ is positive definite if and only if there is a contraction $\Gamma\colon\mathfrak{D}_{T_1}\to\mathfrak{D}_{T_1^{*}}$ such that $T_1^{*}=D_{T_1^{*}} \Gamma D_{T_1} + T_1^{2}.$
	\end{cor}

		\begin{proof}
Since $T(0)=I_\HS$ and $T_1$ is a contraction, one can choose $\Gamma_0=T_1$ in Proposition \ref{+ on Z_3} from which the desired conclusion follows.
		\end{proof}

 The following result is a direct consequence of Proposition \ref{+ on Z_3} and the last corollary.	
	\begin{cor}
		Let $T\colon\mathbb{Z}_{3}\to\mathcal{B}(\mathcal{H})$ be such that $T(0)=I_{\mathcal{H}}$ and $T(1)=T(2)^{*}=T_1$ is a strict contraction. Then $T$ is positive definite if and only if $D_{T_1^{*}}^{-1}(T_1^{*}-T_1^{2}) D_{T_1}^{-1}$ is a contraction.
	\end{cor}
	\vspace{0.2cm}
	
	\subsection{Positive definite functions on groups of order 4}
	
	Every group of order $4$ is abelian and isomorphic to either $\mathbb{Z}_{4}$ or $\mathbb{Z}_{2}\oplus\mathbb{Z}_2$. Thus, it suffices to consider $\mathbb{Z}_4$ and $\mathbb{Z}_{2}\oplus\mathbb{Z}_2$ while studying the positive definite functions on a group of order $4$. First, we need to compute the positive square root of the $ 2 \times 2$ block $\Delta_T=\begin{bmatrix}
		I_\HS & T\\
		T^* & I_\HS
	\end{bmatrix}
	$ for a contraction $T$. By Theorem \ref{2 by 2}, we have that $
	\Delta_T$ is positive if and only if $T$ is a contraction. Hence, $\Delta_T$ has a unique positive square root. Let $S=\begin{bmatrix}
		A & B \\
		C & D \\
	\end{bmatrix}$ be the positive square root of $\Delta_T.$ Since $S$ is self-adjoint, we have $A=A^*,$ $B=C^*$ and $D=D^*.$ Using $S^2=\Delta_T$, we have that
	\begin{equation*}
		\begin{split}
			\begin{bmatrix}
				A^2+BB^* & AB+BD \\
				B^*A + DB^* & B^*B+D^2 \\
			\end{bmatrix}=\begin{bmatrix}
				I_\HS & T\\
				T^* & I_\HS \\
			\end{bmatrix}.
		\end{split}
	\end{equation*}
	Since $I_\HS-BB^*=A^2=AA^* \geq 0,$ we get that $B$ is a contraction. Putting everything together, we have  
	\[
	A=(I_\HS-BB^*)^{1\slash2}=D_{B^*}, \quad  D=(I_\HS-B^*B)^{1\slash2}=D_B,
	\] 
	and
	\[
	T=AB+BD=D_{B^*}B+BD_{B}=2BD_{B}. 
	\]
	Hence, the positive square root of $\Delta_T$ is of the form 
	$
	\begin{bmatrix}
		D_{B^*} & B\\
		B^* & D_B \\
	\end{bmatrix}
	$
	where $T=2BD_{B}$. We write $B=(T)_{1\slash 2}.$ 

	\begin{thm}
		Let $T$ be an $\mathcal{B}(\mathcal{H})$-valued map acting on $\mathbb{Z}_4$ satisfying
		
		\[
		T(0)=I_\HS, \quad  T(1)=T_1=T(3)^{*} \quad  \mbox{and} \quad T(2)=T_2=T(2)^* .
		 \]
		Then the following are equivalent.
		\begin{enumerate}
			\item $T$ is positive definite;
			\item  The block matrix 
			\[
			\Delta_T=\begin{bmatrix}
				I_\HS & T_1 & T_2 & T_1^{*} \\
				T_1^{*} & I_\HS & T_1 & T_2 \\
				T_2 & T_{1}^{*} & I_\HS & T_{1} \\
				T_{1} & T_{2} & T_{1}^{*} & I_\HS\\
			\end{bmatrix} \geq 0;
		\]
		
			
			\item  $T_{1}$ is a  contraction and there is a self adjoint  contraction $\Gamma=\begin{bmatrix}
				\Gamma_{1} & \Gamma_{2} \\
				\Gamma_{2}^{*} & \Gamma_{4}\\
			\end{bmatrix}$ such that 
		\begin{equation}\label{Z_4eq1}
				\begin{split}
				T_{2}&= D_{S^*}\Gamma_{1}D_{S^*} + D_{S^*}\Gamma_{2}S^* + S\Gamma_{2}^*D_{S^*} + S\Gamma_{4}S^* \\
				&=S^{*}\Gamma_{1}S^{*} + S^{*}\Gamma_{2}D_{S} + D_{S}\Gamma_{2}^{*}S + D_{S}\Gamma_{4}D_{S},	
				\end{split}
			\end{equation} 
and 
\begin{equation}\label{Z_4eq2}
T_1=S^*\Gamma_1D_{S^*}+S^*\Gamma_2S^*+D_S\Gamma_2^*D_{S^*}+D_S\Gamma_4 S^*,
\end{equation}
		where, $S=(T_1)_{1\slash2}.$

\item $T_1$ is a contraction, $\pm T_2 \leq I_\HS$ and there are contractions $\Gamma_+ : \overline{(I_\HS+T_2)\mathcal{H}} \to \overline{(I_\HS+T_2)\mathcal{H}}$ and $\Gamma_- : \overline{(I_\HS-T_2)\mathcal{H}} \to \overline{(I_\HS-T_2)\mathcal{H}}$ such that  
\[
T_1+T_1^*=(I_\HS+T_2)^{1\slash 2}\Gamma_+(I_\HS+T_2)^{1\slash 2},
\]
and 
\[
T_1-T_1^*=(I_\HS-T_2)^{1\slash 2}\Gamma_-(I_\HS-T_2)^{1\slash 2}.
\]
		\end{enumerate}
	\end{thm}
		\begin{proof}
		From Example \ref{+ Z_n}, it is clear that $T$ is positive definite if and only if $\Delta_T$ is positive.  We can rewrite $\Delta_T$ as
		\begin{equation}\label{Delta_TZ_4}
		\Delta_T=\begin{bmatrix}
			A & B\\
			B & A
		\end{bmatrix} \ \mbox{for} \ A=\begin{bmatrix}
		I_\HS & T_1 \\
		T_1^* & I_\HS
	\end{bmatrix} \ \mbox{and} \ B=\begin{bmatrix}
	T_2 & T_1^*\\
	T_1 & T_2
\end{bmatrix}=B^*. 
		\end{equation} 
Theorem \ref{2*2I} yields that $\Delta_T \geq 0$ if and only if $A \geq 0$ and there is a contraction $\Gamma$ mapping $\overline{A\mathcal{H}^2}$  into $\overline{A\mathcal{H}^2}$ such that
 \begin{equation}\label{B=A_A}
B=A^{1\slash 2} \Gamma A^{1\slash 2}.
\end{equation}
Since $A$ is a positive operator on $\mathcal{H} \oplus \HS,$ we have that $\overline{Ran} \ A=\overline{Ran} \ A^{1\slash 2}$. So we have
\begin{equation*}
	\begin{split}
		\langle \Gamma A^{1\slash 2}x, A^{1\slash 2} x \rangle 
		=\langle A^{1\slash 2} \Gamma A^{1\slash 2}x, x \rangle
		= \langle Bx,x \rangle 
		= \langle B^*x,x \rangle 
		=\langle A^{1\slash 2} \Gamma^* A^{1\slash 2}x, x \rangle
		= \langle \Gamma^* A^{1\slash 2}x, A^{1\slash 2} x \rangle, 
	\end{split}
\end{equation*}
for every $x \in \mathcal{H}$ and consequently, we have that $\Gamma$ is self-adjoint. Therefore, any contraction $\Gamma$ satisfying (\ref{B=A_A}) can be written in the block matrix form as 
\[
\Gamma=\begin{bmatrix}
	\Gamma_1 & \Gamma_2\\
	\Gamma_2^* & \Gamma_4\\
\end{bmatrix},	
\]
where $\Gamma_1, \Gamma_2$ and $\Gamma_4$ are operators on appropriate spaces such that $\Gamma_1$ and $\Gamma_4$ are self-adjoint.
Again, using Theorem \ref{2*2I} we have that $A \geq 0$ if and only if $T_1$ is a contraction. Also,
\begin{equation*}
	\begin{split}
		A^{1\slash2}\Gamma A^{1\slash 2}
		& =  	\begin{bmatrix}
		D_{S^*} & S\\
		S^* & D_S \\
	\end{bmatrix} \begin{bmatrix}
		\Gamma_1 & \Gamma_2\\
		\Gamma_2^* & \Gamma_4\\
	\end{bmatrix}   	\begin{bmatrix}
	D_{S^*} & S\\
	S^* & D_S \\
\end{bmatrix}\\
& =\small{\begin{bmatrix}
D_{S^*}\Gamma_{1}D_{S^*} + D_{S^*}\Gamma_{2}S^* + S\Gamma_{2}^*D_{S^*} + S\Gamma_{4}S^* & (S^*\Gamma_1D_{S^*}+S^*\Gamma_2S^*+D_S\Gamma_2^*D_{S^*}+D_S\Gamma_4 S^*)^* \\
S^*\Gamma_1D_{S^*}+S^*\Gamma_2S^*+D_S\Gamma_2^*D_{S^*}+D_S\Gamma_4 S^*	&
S^{*}\Gamma_{1}S^{*} + S^{*}\Gamma_{2}D_{S} + D_{S}\Gamma_{2}^{*}S + D_{S}\Gamma_{4}D_{S}
\end{bmatrix}}.\\
	\end{split}
\end{equation*}
Consequently, (\ref{B=A_A}) holds if and only if (\ref{Z_4eq1}) and (\ref{Z_4eq2}) hold. We have proved the equivalence of $(1), (2)$ and $(3)$ and now we prove that $(2) \iff (4).$ From (\ref{Delta_TZ_4}) and Proposition \ref{2*2II}, it follows that 
$\Delta_T \geq 0$ if and only if $A\geq 0$ and $\pm B \leq A.$ We show that these two conditions are equivalent to condition-$(4).$ Clearly, $A \geq 0$ if and only if $T_1$ is contraction. We have the following.
\begin{equation} \label{B<A}
	\begin{split}
		\pm B \leq A & \iff \begin{bmatrix}
			I_\HS -T_2 & T_1-T_1^*\\
			T_1^*-T_1 & I_\HS-T_2\\
		\end{bmatrix} \geq 0 \quad \mbox{and} \quad
	\begin{bmatrix}
		I_\HS +T_2 & T_1+T_1^*\\
		T_1^*+T_1 & I_\HS+T_2\\
	\end{bmatrix} \geq 0. 
	\end{split}
\end{equation}
Using Theorem \ref{2*2I}, we have that the two block matrices in (\ref{B<A}) are positive if and only if  $\pm T_2 \leq I_\HS$ and there are contractions $\Gamma_+ : \overline{(I_\HS+T_2)\mathcal{H}} \to \overline{(I_\HS+T_2)\mathcal{H}}$ and $\Gamma_- : \overline{(I_\HS-T_2)\mathcal{H}} \to \overline{(I_\HS-T_2)\mathcal{H}}$ such that  
\[
T_1+T_1^*=(I_\HS+T_2)^{1\slash 2}\Gamma_+(I_\HS+T_2)^{1\slash 2} \ \mbox{and} \
T_1-T_1^*=(I_\HS-T_2)^{1\slash 2}\Gamma_-(I_\HS-T_2)^{1\slash 2}.
\]
\end{proof}
Next, we characterize the positive definite functions on $\mathbb{Z}_2\oplus \mathbb{Z}_2.$	
	The group $\mathbb{Z}_{2}\oplus\mathbb{Z}_{2}$ is represented as $\{e,a,b,ab\}$, where $ab=ba$ and $a^2=b^2=(ab)^2=e.$  Given an operator-valued function $T$ on $G=\mathbb{Z}_{2}\oplus\mathbb{Z}_{2},$ the block matrix $\Delta_T(G)$ given in Proposition  \ref{+ finite}  takes the following form:
	\[
\Delta_T(G)=\begin{bmatrix}
	T(e) & T(a) & T(b) & T(ab) \\
	T(a) & T(e) & T(ab) & T(b) \\
	T(b) & T(ab) & T(e) & T(a)\\
	T(ab) & T(b) & T(a) & T(e)\\ 
\end{bmatrix}.
	\]
	
	\begin{thm} 
		Let $T\colon G=\mathbb{Z}_{2}\oplus\mathbb{Z}_{2}\to\mathcal{B}(\mathcal{H})$ be such that $T(e)=I_\HS$ and $T(s)=T(s)^{*}$ for all $s\in  G.$ Let $T(a)=T_1, T(b)=T_2$ and $T(ab)=T_3.$ Then the following are equivalent.		\begin{enumerate}
			\item $T$ is a positive definite function;
			\item  The block matrix 
			\[
			\Delta_T=\begin{bmatrix}
				I_\HS & T_1 & T_2 & T_3 \\
				T_1 & I_\HS & T_3 & T_2 \\
				T_2 & T_3 & I_\HS & T_{1} \\
				T_{3} & T_{2} & T_{1} & I_\HS\\
			\end{bmatrix} \geq 0;
		\]
				\item $T_1$ is a contraction, $\pm T_2 \leq I_\HS$ and there are contractions $\Gamma_+ : \overline{(I_\HS+T_2)\mathcal{H}} \to \overline{(I_\HS+T_2)\mathcal{H}}$ and $\Gamma_- : \overline{(I_\HS-T_2)\mathcal{H}} \to \overline{(I_\HS-T_2)\mathcal{H}}$ such that  
				\[
				T_1+T_3=(I_\HS+T_2)^{1\slash 2}\Gamma_+(I_\HS+T_2)^{1\slash 2} \ \mbox{and} \
				T_1-T_3=(I_\HS-T_2)^{1\slash 2}\Gamma_-(I_\HS-T_2)^{1\slash 2}.
				\]
			\end{enumerate}
		\begin{proof}
The equivalence of $(1)$ and $(2)$ follows from Proposition \ref{+ finite} and so we prove $(2) \iff (3).$ We can rewrite $\Delta_T$ as
\begin{equation}\label{Delta_TZ_2+Z_2}
	\Delta_T=\begin{bmatrix}
		A & B\\
		B & A
	\end{bmatrix} \quad \mbox{for} \quad A=\begin{bmatrix}
		I_\HS & T_1 \\
		T_1 & I_\HS
	\end{bmatrix} \quad \mbox{and} \quad B=\begin{bmatrix}
		T_2 & T_3\\
		T_3 & T_2
	\end{bmatrix}=B^*. 
\end{equation} 
Using (\ref{Delta_TZ_2+Z_2}) and Proposition \ref{2*2II}, we have that 
$\Delta_T \geq 0$ if and only if $A\geq 0$ and $\pm B \leq A.$ We show that these two conditions are equivalent to $(3).$ Clearly, $A \geq 0$ if and only if $T_1$ is contraction. We have the following.
\begin{equation} \label{B<<A}
	\begin{split}
		\pm B \leq A & \iff \begin{bmatrix}
			I_\HS -T_2 & T_1-T_3\\
			T_1-T_3 & I_\HS-T_2\\
		\end{bmatrix} \geq 0 \quad \mbox{and} \quad
		\begin{bmatrix}
			I_\HS +T_2 & T_1+T_3\\
			T_1+T_3 & I_\HS+T_2\\
		\end{bmatrix} \geq 0. 
	\end{split}
\end{equation} 
Using Theorem \ref{2*2I}, we have that the two block matrices in (\ref{B<<A}) are positive if and only if  $\pm T_2 \leq I_\HS$ and there are contractions $\Gamma_+ : \overline{(I_\HS+T_2)\mathcal{H}} \to \overline{(I_\HS+T_2)\mathcal{H}}$ and $\Gamma_- : \overline{(I_\HS-T_2)\mathcal{H}} \to \overline{(I_\HS-T_2)\mathcal{H}}$ such that  
\[
T_1+T_3=(I_\HS+T_2)^{1\slash 2}\Gamma_+(I_\HS+T_2)^{1\slash 2} \ \mbox{and} \
T_1-T_3=(I_\HS-T_2)^{1\slash 2}\Gamma_-(I_\HS-T_2)^{1\slash 2}.
\]
The proof is complete.
		\end{proof}
	\end{thm}

\section{Positive Definite Functions on $\mathbb{Z}$ and direct sum of its copies}\label{+ on Z^n}
	
\vspace{0.2cm}		
	
\noindent The positive definite functions on an infinite group such as $\mathbb{Z}^n (n \geq 1)$ have been well-studied in detail, e.g. see \cite{Foias and Frazho, NagyFoias, Nagy} and the references therein. We know that for an operator $P \in \mathcal B(\HS)$, the map 
\begin{equation}\label{eqn6.1}
	m \mapsto \left\{
	\begin{array}{ll}
		P^m & m \geq 1\\
		I_\mathcal{H} & m = 0\\
		P^{*|m|} & m <0\\
	\end{array} 
	\right.
\end{equation}
on $\mathbb{Z}$ is positive definite if and only if $P$ is a contraction, for example, see Section 8.1 in \cite{NagyFoias}. We give an alternative proof to this  statement using the positivity of certain block matrices. There is a canonical extension to the map given in (\ref{eqn6.1}) to the multivariable setting e.g. \cite[Chapter 1]{NagyFoias}. We discuss the two-variable extension of this map for a pair of doubly commuting contractions. The block matrix technique discussed in Section \ref{Sec_+ functions, matrices} helps us to obtain different proofs of several existing results. We conclude this section by deriving Brehmer positivity for a pair of commuting contractions acting on a Hilbert space. 
	\begin{thm}\label{+ on Z}
		For a map $T: \mathbb{Z} \to \mathcal{B}(\mathcal{H})$ the following are equivalent.
		\begin{enumerate}
			\item $T$ is a positive definite function;
			\item The block matrix 
			\[ 
			\Delta_n=\begin{bmatrix}
				
				T(0) & T(1) & T(2) & \cdots  & T(n) \\
				
				T(-1) & T(0) & T(1) & \cdots & T(n-1)\\
				
				T(-2) & T(-1) & T(0) & \cdots & T(n-2) \\
				
				\vdots & \vdots & \vdots & \ddots & \vdots \\
				
				T(-n) & T(-n+1) & T(-n+2) & \cdots & T(0)\\
			\end{bmatrix}
			\]
			defines a positive operator on $\mathcal{H}^{n+1}$ for every $n \in \mathbb{N}.$	
		\end{enumerate}
	\end{thm}
	\begin{proof}
		Note that every function $h \in c_{00}(\mathbb{Z}, \mathcal{H})$ can be written as 
		\[
		h(m)=\overset{n}{\underset{i=-n}{\sum}}h_i\delta_{i}(m) \quad(m \in \mathbb{Z}),
		\]
		for some $n \in \mathbb{N}$ and $h_i \in \mathcal{H}$ where $\delta_i(.)$ is the Kronecker delta map. For the ease of computations, we write 
		\[
		h=(\dotsc, 0, 0, h_{-n}, \dotsc, h_{-1}, \boxed{h_0}, h_1, \dotsc, h_n, 0, 0, \dotsc) \ \mbox{in} \ c_{00}(\mathbb{Z}, \mathcal{H}).
		\]
		 Fix the notation $T(m)=T_m$ for $m \in \mathbb{Z}.$ Then
		\begin{equation*}
			\begin{split}
				\sum_{s,t\in \mathbb{Z}}\langle T(s^{-1}t)h(t),h(s)\rangle 
				&= \overset{n}{\underset{s,t=-n}{\sum}}\langle T(s^{-1}t)h(t),h(s)\rangle\\
				&=\overset{n}{\underset{s=-n}{\sum}}\bigg \langle \overset{n}{\underset{t=-n}{\sum}} T(s^{-1}t)h(t),h(s)\bigg \rangle\\
				&=\bigg\langle 
				\begin{bmatrix}
					h_{-n}\\
					h_{-n+1}\\
					\vdots\\
					h_{n}\\
				\end{bmatrix},
				\begin{bmatrix}
					T_0 & T_1 & \cdots & T_{2n} \\
					T_{-1} & T_{0} & \cdots & T_{2n-1} \\
					\vdots       & \vdots       & \cdots & \vdots\\
					T_{-2n} & T_{-2n+1} & \cdots & T_{0} \\
				\end{bmatrix} \cdot
				\begin{bmatrix}
					h_{-n}\\
					h_{-n+1}\\
					\vdots\\
					h_n\\
				\end{bmatrix}\bigg\rangle,\\
			\end{split}
		\end{equation*}
		which shows that $T$ is positive definite if and only if the block matrix 
		\[
		\Delta_{2n}= 
		\begin{bmatrix}
			T_{0} & T_{1} & \cdots & T_{2n} \\
			T_{-1} & T_{0} & \cdots & T_{2n-1} \\
			\vdots       & \vdots       & \cdots & \vdots\\
			T_{-2n} & T_{-2n+1}  & \cdots & T_{0} \\
		\end{bmatrix} \geq 0
		\]
for every $n \in \mathbb{N}.$ Consequently, if $\Delta_n$ is positive for every $n \in \mathbb{N}$ then $T$ is positive definite. For the converse, assume that $T$ is positive definite. For some $n \in \mathbb{N}$, let $h=(h_0, h_1, \dotsc, h_n) \in \mathcal{H}^{n+1}$ which can be extended to a function in $c_{00}(\mathbb{Z}, \mathcal{H})$ as
		\[
		h(m)=\overset{n+1}{\underset{i=0}{\sum}}h_i\delta_{h_i}(m) \quad(m \in \mathbb{Z}),
		\]
		and again the same calculations give the following:
		\begin{equation*}
			\begin{split}
				\sum_{s,t\in \mathbb{Z}}\langle T(s^{-1}t)h(t),h(s)\rangle 
				&=\bigg\langle 
				\begin{bmatrix}
					h_{0}\\
					h_{1}\\
					\vdots\\
					h_{n}\\
				\end{bmatrix},
				\begin{bmatrix}
					T_0 & T_1 & \cdots & T_{n} \\
					T_{-1} & T_{0} & \cdots & T_{n-1} \\
					\vdots       & \vdots       & \cdots & \vdots\\
					T_{-n} & T_{-n+1} & \cdots & T_{0} \\
				\end{bmatrix} \cdot
				\begin{bmatrix}
					h_{0}\\
					h_{1}\\
					\vdots\\
					h_n\\
				\end{bmatrix}\bigg\rangle, \quad(T(m)=T_m)\\
			\end{split}
		\end{equation*}
		which gives that the block matrix $\Delta_n$ is positive since $T$ is positive definite.
	\end{proof}
We want to prove that the map given in (\ref{eqn6.1}) is positive definite if $P$ is a contraction. To do so, we need the following result.	
	\begin{prop}\label{+ on Z1}
		 An operator $P$ acting on a space $\mathcal{H}$ is a contraction if and only if the block matrix
		\[
		\Delta_n=\begin{bmatrix}
			I_\HS & P & P^2 & \cdots & P^{n} \\
			P^{*} & I_\HS & P & \cdots & P^{n-1} \\
			P^{*2} & P^{*} & I_\HS & \cdots & P^{n-2} \\
			\vdots &	\vdots       & \vdots       & \cdots & \vdots\\
			P^{*n} & P^{*(n-1)} & P^{*(n-2)} & \cdots & I_\HS \\
		\end{bmatrix}
		\]
		is a positive operator on $\mathcal{H}^{n+1}$ for every $n \in \mathbb{N}.$
	\end{prop}
	
\begin{proof}
A proof to this Proposition is given as a part of the proof of Theorem 2.6 in \cite{Paulsen}. We imitate the proof here for the sake of completeness. Assume that $P$ is a contraction and set 
		\[
		Q=\begin{bmatrix}
			0 & 0 & 0 & \cdots & 0 & 0 \\
			P^* & 0 & 0 & \cdots & 0 & 0 \\
			0 & P^* & 0 & \cdots & 0 & 0 \\
			\vdots &	\vdots       & \vdots       & \cdots & \vdots & \vdots\\
			0 & 0 & 0 & \cdots & P^* & 0 \\
		\end{bmatrix}_{(n+1) \times (n+1)}.
		\]
		Note that $Q^{(n+1)}=0$ and $\|Q\| \leq \|P^*\| \leq 1$. Let $I_n$ be the indentity operator on $\mathcal{H}^{n+1}$. We see that $I_n-Q$ is invertible with the inverse given by
		\[
			(I_n-Q)^{-1}=\begin{bmatrix}
				I_\HS & 0 & 0 & \cdots & 0 & 0 \\
				P^* & I_\HS & 0 & \cdots & 0 & 0 \\
				P^{*2} & P^* & I_\HS & \cdots & 0 & 0 \\
				\vdots &	\vdots       & \vdots       & \cdots & \vdots & \vdots\\
				P^{*(n-1)} & P^{*(n-2)} & P^{*(n-3)} & \cdots & P^* & I_\HS \\
			\end{bmatrix}.
			\]
			Consequently, we have 
			\[
			\Delta_n=I_n+Q+Q^2+\dotsc +Q^n+Q^*+Q^{*2}+ \dotsc +Q^{*n}=(I_n-Q)^{-1}+(I_n-Q^*)^{-1}-I_n.
			\] 
			To see that $\Delta_n$ is positive, fix $h \in \mathcal{H}^{n+1}$ and let $h=(I_n-Q)y$ for some $y \in \mathcal{H}^{n+1}.$ Then 
			\begin{equation*}
				\begin{split}
					\langle \Delta_n h,h\rangle
					&=\langle ((I_n-Q)^{-1}+(I_n-Q^*)^{-1}-I_n) h,h\rangle\\ 
					&=\langle y, (I_n-Q)y\rangle + \langle (I_n-Q)y, y \rangle -\langle (I_n-Q)y, (I_n-Q)y\rangle\\
					&=\|y\|^2-\|Qy\|^2 \geq 0,
				\end{split}
			\end{equation*}
			as $Q$ is a contraction. Since $n \in \mathbb{N}$ is arbitrary, we have that $\Delta_n$ is positive for every $n \in \mathbb{N}.$ For the converse, we see that if $\Delta_1$ is positive, then the block matrix
			$
			\begin{bmatrix}
				I_\HS & P\\
				P^* & I_\HS
			\end{bmatrix}
			$
			is positive which implies that $P$ is a contraction since for any $h \in \mathcal{H},$ we have that 
			\[
			\|h\|^2-\|Ph\|^2=\bigg  \langle \begin{bmatrix}
				I_\HS & P\\
				P^* & I_\HS
			\end{bmatrix} \cdot \begin{bmatrix}
				-Ph\\
				h
			\end{bmatrix}, 
			\begin{bmatrix}
				-Ph\\
				h
			\end{bmatrix}\bigg \rangle \geq 0.
			\]
			The proof is complete.
		\end{proof}
		\begin{cor}\label{particular example Z}
			Given an operator $P$ acting on a Hilbert space $\HS$, the map $T: \mathbb{Z} \to \mathcal{B}(\mathcal{H})$ as in $(\ref{eqn6.1})$
			is positive definite if and only if $P$ is a contraction.
		\end{cor}
		\begin{proof}
			From Theorem \ref{+ on Z}, it follows that $T$ is positive definite if and only if the block matrix 
			\[ 
			\Delta_n=\begin{bmatrix}
				T(0) & T(1) & T(2) & \cdots  & T(n) \\
				T(-1) & T(0) & T(1) & \cdots & T(n-1)\\
				T(-2) & T(-1) & T(0) & \cdots & T(n-2) \\
				\vdots & \vdots & \vdots & \ddots & \vdots \\
				T(-n) & T(-n+1) & T(-n+2) & \cdots & T(0)\\
			\end{bmatrix}
			=\begin{bmatrix}
				I_\HS & P & P^2 & \cdots & P^{n} \\
				P^{*} & I_\HS & P & \cdots & P^{n-1} \\
				P^{*2} & P^{*} & I_\HS & \cdots & P^{n-2} \\
				\vdots &	\vdots       & \vdots       & \cdots & \vdots\\
				P^{*n} & P^{*(n-1)} & P^{*(n-2)} & \cdots & I_\HS \\
			\end{bmatrix}
			\]
			defines a positive operator on $\mathcal{H}^{n+1}$ for every $n \in \mathbb{N}$. Hence, the desired conclusion follows from Proposition \ref{+ on Z1}.
		\end{proof}
		 Now, we find associated block matrices for the group $\mathbb{Z}\oplus \mathbb{Z}.$ To write the block matrix explicitly, we need to define an order on the elements of $\mathbb{Z}\oplus \mathbb{Z}$ for which we choose the \textit{lexicographic order} given by 
		\[
		(i_1, j_1) <  (i_2, j_2) \quad \mbox{if either} \quad i_1 < i_2 \quad \mbox{or } \quad i_1=i_2 \quad \mbox{and} \quad j_1<j_2,
		\]
		for every $(i_1, j_1), (i_2, j_2) \in \mathbb{Z}\oplus \mathbb{Z}.$ Then any $\hat{h} \in c_{00}(\mathbb{Z}\oplus \mathbb{Z}, \mathcal{H})$ can be written as 
		\begin{equation*}
			\hat{h}(i,j)=\overset{(n,n)}{\underset{(k,l)=(-n,-n)}{\sum}}h_{k,l}\delta_{(k,l)}(i,j) \quad(k,l \in \mathbb{Z})    
		\end{equation*}
		for some $n \in \mathbb{N}$ and $h_{k,l} \in \mathcal{H}.$ We simply write 
		\[
		\hat{h}=\bigg(h(-n,-n), \dotsc, h(-n,n), h(-n+1, -n), \dotsc, h(-n+1, n) \dotsc, h(n,-n), \dotsc, h(n,n) \bigg),
		\]
		in  $c_{00}(\mathbb{Z}\oplus \mathbb{Z}, \mathcal{H}).$ For example, if $n=1$ then 
		\[
		\hat{h}=\bigg(h(-1, -1), h(-1, 0), h(-1, 1), h(0,-1), h(0,0), h(0,1), h(1,-1), h(1, 0), h(1,1) \bigg).
		\]
		Now, we compute the following. Fix the notation $T(i,j)=T_{i,j}$ for $i,j\in \mathbb{Z}.$ For $s=(s_1,s_2)$ and $t=(t_1, t_2),$ we have that
		\begin{equation*}
			\begin{split}
				\sum_{s,t\in \mathbb{Z}\oplus \mathbb{Z}}\langle T(s^{-1}t)\hat{h}(t),\hat{h}(s)\rangle 
				&= \overset{n}{\underset{s_1,t_1=-n}{\sum}} \  \overset{n}{\underset{s_2,t_2=-n}{\sum}}
				\bigg \langle T(t_1-s_1, t_2-s_2)\hat{h}(t_1, t_2),\hat{h}(s_1, s_2)
				\bigg \rangle,\\
			\end{split}
		\end{equation*} 
		which is equal to
		\begin{small} \[
			\bigg\langle 
			\left[
			\begin{array}{cccc|c|cccc}
				T_{0, 0} & T_{0, 1} & \cdots & T_{0,2n}  & \cdots & T_{2n,0} & T_{2n,1} & \cdots & T_{2n, 2n}  \\
				T_{0,-1} & T_{0, 0} & \cdots & T_{0,2n-1} & \cdots & T_{2n,-1} & T_{2n,0} & \cdots & T_{2n,2n-1} \\
				\vdots       & \vdots       & \cdots & \vdots & \vdots       & \vdots       & \vdots & \cdots  & \vdots\\
				T_{0,-2n} & T_{0,-2n+1} & \cdots & T_{0,0} & \cdots & T_{2n,-2n} & T_{2n,-2n+1} & \cdots & T_{2n,0} \\
				\hline
				\vdots       & \vdots       & \cdots & \vdots & \vdots & \vdots       & \vdots       & \cdots & \vdots\\
				\hline
				T_{-2n,0} & T_{-2n,1} & \cdots & T_{-2n,n} & \cdots & T_{0, 0} & T_{0, 1} & \cdots & T_{0, 2n}  \\
				T_{-2n,-1} & T_{-2n,0} & \cdots & T_{-2n,2n-1} & \cdots & T_{0,-1} & T_{0, 0} & \cdots & T_{0,2n-1} \\
				\vdots       & \vdots       & \cdots & \vdots & \vdots       & \vdots       & \cdots & \vdots & \vdots \\
				T_{-2n,-2n} & T_{-2n,-2n+1} & \cdots & T_{-2n,0} & \cdots & T_{0,-2n} & T_{0,-2n+1} & \cdots & T_{0, 0} \\	
			\end{array} 
			\right]
			\begin{bmatrix}
				h_{-n,-n}\\
				h_{-n, -n+1}\\
				\vdots \\
				h_{-n,n}\\
				\vdots \\
				h_{n,-n}\\
				h_{n, -n+1}\\
				\vdots \\
				h_{n,n}\\
			\end{bmatrix},
			\begin{bmatrix}
				h_{-n,-n}\\
				h_{-n, -n+1}\\
				\vdots \\
				h_{-n,n}\\
				\vdots \\
				h_{n,-n}\\
				h_{n, -n+1}\\
				\vdots \\
				h_{n,n}\\
			\end{bmatrix}
			\bigg\rangle .
			\]
		\end{small} 
		$ $\\
		This shows that $T$ is positive definite if and only if the block matrix 
		\[
		\Delta_{2n+1}= 
		\left[
		\begin{array}{cccc|c|cccc}
			T_{0, 0} & T_{0, 1} & \cdots & T_{0,2n}  & \cdots & T_{2n,0} & T_{2n,1} & \cdots & T_{2n, 2n}  \\
			T_{0,-1} & T_{0, 0} & \cdots & T_{0,2n-1} & \cdots & T_{2n,-1} & T_{2n,0} & \cdots & T_{2n,2n-1} \\
			\vdots       & \vdots       & \cdots & \vdots & \vdots       & \vdots       & \vdots & \cdots & \vdots\\
			T_{0,-2n} & T_{0,-2n+1} & \cdots & T_{0,0} & \cdots & T_{2n,-2n} & T_{2n,-2n+1} & \cdots & T_{2n,0} \\
			\hline
			\vdots       & \vdots       & \cdots & \vdots & \vdots & \vdots       & \vdots       & \cdots & \vdots\\
			\hline
			T_{-2n,0} & T_{-2n,1} & \cdots & T_{-2n,n} & \cdots & T_{0, 0} & T_{0, 1} & \cdots & T_{0, 2n}  \\
			T_{-2n,-1} & T_{-2n,0} & \cdots & T_{-2n,2n-1} & \cdots & T_{0,-1} & T_{0, 0} & \cdots & T_{0,2n-1} \\
			\vdots       & \vdots       & \cdots & \vdots & \vdots       & \vdots       & \vdots & \cdots & \vdots \\
			T_{-2n,-2n} & T_{-2n,-2n+1} & \cdots & T_{-2n,0} & \cdots & T_{0,-2n} & T_{0,-2n+1} & \cdots & T_{0,0} \\	
		\end{array} 
		\right]
		\]
		is a positive operator on $\mathcal{H}^{(2n+1)^2}$ for every $n \in \mathbb{N}.$ Consequently, if $\Delta_{n+1}$ is positive for every $n \in \mathbb{N},$ then $T$ is positive definite. Conversely, suppose $T$ is positive definite and let $n \in \mathbb{N}.$ Let $\hat{h}=(h_0, h_1, \dotsc, h_n)$ for $h_i=(h_{i0}, h_{i1}, \dotsc, h_{in})$ in $\mathcal{H}^{(n+1)^2}$ which can be extended to a function in $c_{00}(\mathbb{Z}\oplus \mathbb{Z}, \mathcal{H})$ as 
		\[
		\hat{h}(i,j)=\overset{(n,n)}{\underset{(k,l)=(0,0)}{\sum}}h_{k,l}\delta_{(k,l)}(i,j) \quad(k,l \in \mathbb{Z}).
		\]
		Using the same calculations, we have that 
		\[
		\langle \Delta_{n+1} \hat{h}, \hat{h} \rangle =  \overset{n}{\underset{s_1,t_1=0}{\sum}} \  \overset{n}{\underset{s_2,t_2=0}{\sum}}
		\bigg \langle T(t_1-s_1, t_2-s_2)\hat{h}(t_1, t_2),\hat{h}(s_1, s_2)
		\bigg \rangle,\\
		\]
		which is positive since $T$ is a positive definite function. Since $n \in \mathbb{N}$ is arbitrary, we have proved the following result.
		\begin{thm}\label{+ on Z+Z}
			Let $T: \mathbb{Z} \oplus \mathbb{Z} \to \mathcal{B}(\mathcal{H})$ be an operator-valued function. 
			Then the following are equivalent.
			\begin{enumerate}
				\item $T$ is a positive definite function;
				\item The block matrix 
				\begin{equation}\label{Z+Z eq}
					\Delta_{n+1}=\left[
					\begin{array}{cccc|c|cccc}
						T_{0, 0} & T_{0, 1} & \cdots & T_{0,n}  & \cdots & T_{n,0} & T_{n,1} & \cdots & T_{n, n}  \\
						T_{0,-1} & T_{0, 0} & \cdots & T_{0,n-1} & \cdots & T_{n,-1} & T_{n,0} & \cdots & T_{n,n-1} \\
						\vdots       & \vdots       & \cdots & \vdots & \vdots       & \vdots       & \vdots & \cdots  & \vdots \\
						T_{0,-n} & T_{0,-n+1} & \cdots & T_{0,0} & \cdots & T_{n,-n} & T_{n,-n+1} & \cdots & T_{n,0} \\
						\hline
						\vdots       & \vdots       & \cdots & \vdots & \vdots & \vdots       & \vdots       & \cdots & \vdots\\
						\hline
						T_{-n,0} & T_{-n,1} & \cdots & T_{-n,n} & \cdots & T_{0, 0} & T_{0, 1} & \cdots & T_{0, n}  \\
						T_{-n,-1} & T_{-n,0} & \cdots & T_{-n,n-1} & \cdots & T_{0,-1} & T_{0, 0} & \cdots & T_{0,n-1} \\
						\vdots       & \vdots       & \cdots & \vdots & \vdots       & \vdots       & \vdots & \cdots & \vdots \\
						T_{-n,-n} & T_{-n,-n+1} & \cdots & T_{-n,0} & \cdots & T_{0,-n} & T_{0,-n+1} & \cdots & T_{0,0} \\	
					\end{array} 
					\right]
				\end{equation} 
				defines a positive operator on $\mathcal{H}^{(n+1)^2}$ for every $n \in \mathbb{N}.$	
			\end{enumerate}
		\end{thm}
		\begin{rem}\label{delta for Z+Z} For a pair of commuting operators $(T_1, T_2)$ acting on $\mathcal{H}$, we consider the operator valued-function $T: \mathbb{Z}\oplus \mathbb{Z} \to \mathcal{B}(\mathcal{H})$ given by
			\begin{equation}\label{dcmap}
			T(m,n):=\left\{
			\begin{array}{ll}
				T_1^mT_2^n & m,n \geq 0\\
				T_1^{*|m|}T_2^n & m <0, n >0\\
				T_2^{*|n|}T_1^m & m >0, n <0\\
				T_1^{*|m|}T_2^{*|n|} & m, n <0 \ .\\
			\end{array} 
			\right.
			\end{equation}
						The associated block matrix $\Delta_n$ given by $(\ref{Z+Z eq})$ has the following form.
			
			\begin{equation}\label{Z+Z particular}
				\Delta_{n+1}=\left[
				\begin{array}{cccc|c|cccc}
					I_\HS & T_2 & \cdots & T_2^n  & \cdots & T_1^{n} & T_1^{n}T_2 & \cdots & T_1^nT_2^n  \\
					T_2^* & I_\HS & \cdots & T_2^{n-1} & \cdots & T_2^*T_1^n & T_1^n & \cdots & T_1^nT_2^{n-1} \\
					\vdots       & \vdots       & \cdots & \vdots & \vdots       & \vdots       & \vdots & \cdots & \vdots\\
					T_2^{*n} & T_2^{*(n-1)} & \cdots & I_\HS & \cdots & T_2^{*n}T_1^n & T_2^{*(n-1)}T_1^n & \cdots & T_1^n \\
					\hline
					\vdots       & \vdots       & \cdots & \vdots & \vdots & \vdots       & \vdots       & \cdots & \vdots\\
					\hline
					T_1^{*n} & T_1^{*n}T_2 & \cdots & T_1^{*n}T_2^n & \cdots & I_\HS & T_2 & \cdots & T_2^n  \\
					T_2^*T_1^{*n} & T_1^{*n} & \cdots & T_1^{*n}T_2^{n-1} & \cdots & T_2^* & I_\HS & \cdots & T_2^{n-1} \\
					\vdots       & \vdots       & \cdots & \vdots & \vdots       & \vdots       & \vdots & \cdots & \vdots\\
					T_2^{*n}T_1^{*n} & T_2^{*(n-1)}T_1^{*n} & \cdots & T_1^{*n} & \cdots & T_2^{*n} & T_2^{*(n-1)} & \cdots & I_\HS \\	
				\end{array} 
				\right],
			\end{equation} 
			which can be re-written as
			\begin{equation}\label{Recursive}
				\Delta_{n+1}=\begin{bmatrix}
					A_{11} & A_{12} & \cdots & A_{1,n+1} \\
					A_{21} & A_{22} & \cdots & A_{2,n+1} \\
					\vdots & \vdots & \cdots &  \vdots\\
					A_{n+1, 1} & A_{n+1, 2} & \cdots & A_{n+1, n+1} \\
				\end{bmatrix}
				\ \mbox{such that} \ A_{ij}=A_{i+1, j+1} \ \mbox{and} \ A_{ji}=A_{ij}^*,
			\end{equation}
			for all $i,j$ and each $A_{ij}$ is an $(n+1) \times (n+1)$ block matrix with entries consisting of operators on $\mathcal{H}.$ It is evident from $(\ref{Recursive})$ that $\Delta_n$ is completely determined by $A_{11}, A_{12}, \dotsc, A_{1, n+1}$. By $(\ref{Z+Z particular})$, we have that $A_{1j}$ is equal to
			
			\begin{equation}\label{A_1j}
				\begin{bmatrix}
					T_1^{j-1} & T_1^{j-1}T_2 & \cdots & T_1^{j-1}T_2^n  \\
					T_2^*T_1^{j-1} & T_1^{j-1} & \cdots & T_1^{j-1}T_2^{n-1} \\
					\vdots & \vdots   & \cdots & \vdots\\
					T_2^{*n}T_1^{j-1} & T_2^{*(n-1)}T_1^{j-1} & \cdots & T_1^{j-1} \\
				\end{bmatrix}
				=
				\begin{bmatrix}
					I_\HS & T_2 & \cdots & T_2^n  \\
					T_2^* & I_\HS & \cdots & T_2^{n-1} \\
					\vdots & \vdots   & \cdots & \vdots\\
					T_2^{*n} & T_2^{*(n-1)} & \cdots & I_\HS \\
				\end{bmatrix} 
				\begin{bmatrix}
					T_1 & O & \cdots & O  \\
					O & T_1 & \cdots & O\\
					\vdots & \vdots   & \cdots & \vdots\\
					O & O & \cdots  & T_1\\
				\end{bmatrix}^{j-1},
			\end{equation}
			for $1 \leq j \leq n+1.$ Consequently, 
			\begin{equation}
				A_{1j}=A_{11}\Lambda_{n+1}^{j-1} \quad \mbox{for}  \quad \Lambda_{n+1}=  \begin{bmatrix}
					T_1 & O & \cdots & O  \\
					O & T_1 & \cdots & O\\
					\vdots & \vdots   & \cdots & \vdots\\
					O & O & \cdots  & T_1\\
				\end{bmatrix} \qquad ( 1\leq j \leq n+1).
			\end{equation} 
		\end{rem}
	One can extend the map given in (\ref{eqn6.1}) to $\mathbb{Z}^n$ for a commuting $n$-tuple of contractions and similar conclusions can be drawn as in Corollary \ref{particular example Z}. We prove a few results in this direction in the two variable setting. There are several proofs available in the literature (e.g. see \cite{Nagy, Paulsen}) of the following results. We provide a different proof of these results using the positivity of the block matrices. Recall that a pair $(T_1, T_2)$ is said to be \textit{doubly commuting} if $T_1T_2=T_2T_1$ and $T_1^*T_2=T_2T_1^*$. 
		\begin{prop}\label{example + on Z+Z}
			For a pair of doubly commuting operators $(T_1, T_2)$ acting on a Hilbert space $\mathcal{H}$, the map $T: \mathbb{Z}\oplus \mathbb{Z} \to \mathcal{B}(\mathcal{H})$ given as in $(\ref{dcmap})$ is positive definite if and only if $T_1$ and $T_2$ are contractions.
		\end{prop}
		\begin{proof}
			It follows from Proposition $\ref{+ on Z1}$ that $A_{11}$ in $(\ref{A_1j})$ is positive if and only if $T_2$ is a contraction. Assume that $T_1$ and $T_2$ are contractions. By Remark \ref{delta for Z+Z} it suffices to prove that $\Delta_{n+1}$ is positive for every $n \in \mathbb{N}.$ Since $T_1$ and $T_2$ doubly commute, we have $A_{11}\Lambda_{n+1}=\Lambda_{n+1}A_{11}$. Also, $A_{11}\Lambda_{n+1}^*=\Lambda_{n+1}^*A_{11}$ as $A_{11}$ is a positive operator. So, we have that
			
			\begin{equation*}
				\begin{split}
					\Delta_{n+1} 
					&=
					\begin{bmatrix}
						A_{11} & A_{12} & \cdots & A_{1,n+1} \\
						A_{21} & A_{22} & \cdots & A_{2,n+1} \\
						\vdots & \vdots & \cdots &  \vdots\\
						A_{n+1, 1} & A_{n+1, 2} & \cdots & A_{n+1, n+1} \\
					\end{bmatrix} \\
					&= 
					\begin{bmatrix}
						A_{11} & A_{11}\Lambda_{n+1} & \cdots & A_{11}\Lambda_{n+1}^n \\
						A_{11}\Lambda_{n+1}^* & A_{11} & \cdots & A_{11}\Lambda_{n+1}^{n-1} \\
						\vdots & \vdots & \cdots &  \vdots\\
						A_{11}\Lambda_{n+1}^{*n} & A_{11}\Lambda_{n+1}^{*(n-1)} & \cdots & A_{11} \\
					\end{bmatrix} \ \mbox{[By (\ref{A_1j})]}\\
					&=
					\begin{bmatrix}
						A_{11}^{1\slash2} & O & \cdots & O  \\
						O &  A_{11}^{1\slash2} & \cdots & O\\
						\vdots & \vdots   & \cdots & \vdots\\
						O & O & \cdots  &  A_{11}^{1\slash2}\\
					\end{bmatrix}
					\begin{bmatrix}
						I_\HS & \Lambda_{n+1} & \cdots & \Lambda_{n+1}^n \\
						\Lambda_{n+1}^* & I_\HS & \cdots & \Lambda_{n+1}^{n-1} \\
						\vdots & \vdots & \cdots &  \vdots\\
						\Lambda_{n+1}^{*n} & \Lambda_{n+1}^{*(n-1)} & \cdots & I_\HS \\
					\end{bmatrix} 
					\begin{bmatrix}
						A_{11}^{1\slash2} & O & \cdots & O  \\
						O &  A_{11}^{1\slash2} & \cdots & O\\
						\vdots & \vdots   & \cdots & \vdots\\
						O & O & \cdots  &  A_{11}^{1\slash2}\\
					\end{bmatrix}.
				\end{split}
			\end{equation*}
			Since $T_1$ is a contraction and $\Lambda_{n+1}$ is a contraction, Proposition $\ref{+ on Z}$ yields that the block matrix
			\[
			\begin{bmatrix}
				I_\HS & \Lambda_{n+1} & \cdots & \Lambda_{n+1}^n \\
				\Lambda_{n+1}^* & I_\HS & \cdots & \Lambda_{n+1}^{n-1} \\
				\vdots & \vdots & \cdots &  \vdots\\
				\Lambda_{n+1}^{*n} & \Lambda_{n+1}^{*(n-1)} & \cdots & I_\HS \\
			\end{bmatrix}
			\]
			is positive. Consequently, we have that $\Delta_{n+1}$ is positive for every $n\in \mathbb{N}$ and from Theorem $\ref{+ on Z+Z}$, we have that $T$ is a positive definite function.
			Conversely, if $T$ is a positive definite function then we have  
			\[
			T_1(m):=T(m,0)=\left\{
			\begin{array}{ll}
				T_1^m & m \geq 1\\
				I_\HS & m = 0\\
				T_1^{*|m|} & m <0\\
			\end{array} 
			\right. \quad{\mbox{and}} \quad 
			T_2(n):=T(0,n)=\left\{
			\begin{array}{ll}
				T_2^n & n \geq 1\\
				I_\HS & n = 0\\
				T_2^{*|n|} & n <0\\
			\end{array} 
			\right., 
			\]
			are positive definite functions on $\mathbb{Z}.$ By Corollary \ref{particular example Z}, $T_1$ and $T_2$ are contractions.
		\end{proof}
		 Next, we prove that the map $T$ as in $(\ref{dcmap})$ satisfies Brehmer positivity for a pair of commuting operators if $T$ is positive definite. Recall that a pair of commuting operators $(T_1, T_2)$ acting on a Hilbert space $\HS$ is said to satisfy Brehmer positivity if $T_1, T_2$ are contractions and $I_\HS-T_1^*T_1-T_2^*T_2+(T_1T_2)^*T_1T_2 \geq 0$.
		\begin{prop}
			Let $(T_1, T_2)$ be a pair of commuting operators  acting on a Hilbert space $\mathcal{H}$. If the map $T: \mathbb{Z}\oplus \mathbb{Z} \to \mathcal{B}(\mathcal{H})$ given as in $(\ref{dcmap})$ is positive definite, then $(T_1, T_2)$ satisfies Brehmer positivity.
		\end{prop}
		\begin{proof}
			Since $T$ is positive definite, Theorem \ref{+ on Z+Z} and Remark \ref{delta for Z+Z} yield that the block matrix 
			\begin{equation*}
				\Delta_2=\begin{bmatrix}
					I_\HS & T_2 & T_1 & T_1T_2\\
					T_2^* & I_\HS & T_2^*T_1 & T_1\\
					T_1^* & T_1^*T_2 & I_\HS & T_2\\
					T_1^*T_2^* & T_1^* & T_2^* & I_\HS\\
				\end{bmatrix}
			\end{equation*}
			is positive. Given $h \in \mathcal{H},$ define $x=\begin{bmatrix}
				T_1T_2h & -T_1h & -T_2h & h
			\end{bmatrix}^t$ in $\mathcal{H}^4$ for which we have that
			\begin{equation*}
				\begin{split}
					0 \leq \langle \Delta_2 x, x \rangle
					&=
					\bigg \langle \begin{bmatrix}
						I_\HS & T_2 & T_1 & T_1T_2\\
						T_2^* & I_\HS & T_2^*T_1 & T_1\\
						T_1^* & T_1^*T_2 & I_\HS & T_2\\
						T_1^*T_2^* & T_1^* & T_2^* & I_\HS\\
					\end{bmatrix}
					\begin{bmatrix}
						T_1T_2h\\
						-T_1h\\
						-T_2h\\
						h
					\end{bmatrix},  \begin{bmatrix}
						T_1T_2h\\
						-T_1h\\
						-T_2h\\
						h
					\end{bmatrix}
					\bigg \rangle  \\
					&=    
					\bigg \langle 
					\begin{bmatrix}
						0\\
						0\\
						0\\
						T_1^*T_2^*T_1T_2h-T_1^*T_1h-T_2^*T_2h+h
					\end{bmatrix},  \begin{bmatrix}
						T_1T_2h\\
						-T_1h\\
						-T_2h\\
						h
					\end{bmatrix}
					\bigg \rangle  \\
					&=\bigg \langle (I_\HS-T_1^*T_1-T_2^*T_2+(T_1T_2)^*T_1T_2)h, h \bigg \rangle .
				\end{split}
			\end{equation*}
			Since $h \in \mathcal{H}$ is arbitrary, we have that $I_\HS-T_1^*T_1-T_2^*T_2+(T_1T_2)^*T_1T_2 \geq 0.$ As seen previously in the proof of Proposition \ref{example + on Z+Z}, we have that $T_1$ and $T_2$ are contractions.
		\end{proof}
		

\section{Structure Theorem for Unitary representations on a finite group}\label{structure thm}

\vspace{0.3cm}		
	
\noindent We have seen that Naimark's Theorem relates the positive definite functions to the unitary representations and vice-versa. In the previous sections, we have characterized the positive definite functions on some finite and infinite groups. In this Section, we characterize the unitary representations acting on a finite group such that the image consists of commuting unitaries. Such a unitary representation is called a commutative unitary representation. We begin with an elementary result which is a consequence of spectral theorem.

\begin{thm}\label{Unitary}
Let $\underline{U}=(U_1,\dotsc, U_m)$ be a tuple of commuting operators acting on
a Hilbert space $\mathcal{H}$ such that its Taylor-joint spectrum, $\sigma_{T}(\underline{U})$ contains $k$ many elements in $\C^m$. Then $\underline{U}$ consists of unitary operators if and only if there is a
commuting $k$-tuple of mutually orthogonal projections
$P=\left(P_1, \dotsc, P_k\right)$ and scalars $\{\alpha_{ij} \colon\;1\leq i\leq
k, 1\leq j\leq m\} \subseteq \mathbb{T}$ such that $P_1+ \dotsc + P_k=I_\HS$ and
\begin{equation}
	U_j= \sum_{i=1}^{k}\alpha_{ij}P_{i},
\end{equation}
for $j=1, \dotsc, m$.
\end{thm}
\begin{lem}\label{Unitary_Spectrum}
	Let $U\colon G\to\mathcal{B}(\mathcal{H})$ be a unitary representation acting on a finite group $G$ of order $m$. Then the spectrum of $U(s)$ is a subset of the group
	$\mu_m=\{z \in \mathbb{C}\ :\ z^m=1\}$ for every $s \in G$.
\end{lem}
	\begin{proof}
		Since $U$ is a group homomorphism, it follows that $U(s)^m=U(s^m)=U(e)=I_\HS$ for every $s \in G$. The desired conclusion follows from the spectral mapping theorem (see .
	\end{proof}
The following is the main result of this section.
\begin{thm}\label{Unitary_Srtucture1}
	Let $G$ be a finite group of order $m$ and $U\colon G\to\mathcal{B}(\mathcal{H})$ be an operator-valued function. Then $U$ is a commutative unitary representation if and only
	if there are finitely many orthogonal projections
	$P_1, \dotsc,P_k$ and scalars $\{\lambda_i(s)\in\mathbb{T}: 1\leq i\leq k, s\in G\}$ such that
	\begin{enumerate}
		\item $P_1 + \dotsc +P_k=I_\HS$ (resolution of identity);
		\vspace{0.15cm}
	\item $U(s)=\sum_{i=1}^{k}\lambda_i(s)P_i$  for all $s \in G;$
	\vspace{0.15cm}
	\item $\lambda_{i}(st)=\lambda_{i}(s)\lambda_{i}(t)$ for all $s,t\in G$ and $1\leq
	i\leq k.$
\end{enumerate}
	Moreover,	$\sigma(U(s))=\{\lambda_1(s),\dotsc,\lambda_k(s)\}$ for every $s\in G.$
\end{thm}
\begin{proof}
		Let $G=\{g_1,\dotsc,g_m\}$ with the identity $g_1=e.$ Let $U$ be a commutative unitary representation. Since each $U(g_j)$ is a unitary, it follows from Lemma \ref{Unitary_Spectrum} that 
		\[
		\sigma_T(U(g_1), \dotsc, U(g_m)) \subseteq \sigma(U(g_1)) \dotsc \times \sigma(U(g_m)) \subseteq \mu_m  \times \dotsc \times \mu_m \ \mbox{($m$-times)},		
		\]
which is a finite subset of $\mathbb{T}^m.$ By Corollary \ref{Unitary}, there are  mutually orthogonal projections
		$P_1, \dotsc, P_k$ and scalars $\{\alpha_{ij} \colon\;\alpha_{ij}\in\mathbb{T}
		\mbox{ for all }1\leq i\leq k, 1\leq j\leq m\}$ for some $k \in \mathbb{N}$  such that
		
		\begin{equation}\label{Equn3}
			U(g_j)=\sum_{i=1}^{k}\alpha_{ij}P_{i} \quad \text{ and
			} \quad \sigma(U(g_j))=\{\alpha_{1j}, \dotsc,\alpha_{kj}\} \quad \mbox{($
				1 \leq j \leq m$) }.
		\end{equation}
For $1\leq i\leq k$, define $\lambda_i\colon G\to\mathbb{T}$ by
		$\lambda_i(g_j)=\alpha_{ij}$ for all $1\leq j\leq m. $ Then (\ref{Equn3})
		can be written as
		
		\begin{equation}\label{Equn4}
			U(g_j)=\sum_{i=1}^{k}\lambda_i(g_j) P_{i} \qquad \text{and} \qquad \sigma(U(g_j))=\{\lambda_1(g_j),\dotsc,\lambda_k(g_j)\} 
		\end{equation}
		for $1 \leq j \leq m$.
Since $P_i's$ are mutually orthogonal projections and each $|\lambda_i(g_j)|=1$, we have 

		\begin{equation*}
			\begin{split}
				I_\HS=U(g_j)U(g_j)^* =\left(\sum_{i=1}^{k}\lambda_{i}           
				(g_j)P_{i}\right)\left(\sum_{i=1}^{k}\overline{\lambda_{i}(g_j)}P_{i}\right)
				=\sum_{i=1}^{k}\vert{\lambda}_{i}(g_j)\vert ^2 P_{i}
				= P_1+\dotsc+P_k,
			\end{split}
		\end{equation*}
for every $j=1, \dotsc, k.$ It only remains to show that each $\lambda_i$ defines a homomorphism. Since $U$ is a group homomorphism, we have that 
\begin{equation}\label{lambda_hom}
	\begin{split}
	\sum_{i=1}^{k}\lambda_i(st) P_{i}
	=U(st) =U(s)U(t)
	=\bigg(\sum_{i=1}^{k}\lambda_i(s) P_{i}\bigg)	\cdot \bigg(\sum_{i=1}^{k}\lambda_i(t) P_{i}\bigg)
	=\sum_{i=1}^{k}\lambda_i(s)\lambda_l(t) P_{i}, \ \ \ \ \ 
	\end{split}
\end{equation}
for every $s,t \in G.$ Given $j$, pre-multiplying (\ref{lambda_hom}) with $P_j$, we get that
\[
\lambda_j(st)P_j=\lambda_j(s)\lambda_j(t)P_j,
\]
for every $s,t \in G.$ Since $P_j$'s are non-zero projections, we have that $\lambda_j(st)=\lambda_j(s)\lambda_j(t)$ for every $j=1, \dotsc, k$ and $s,t \in G.$ 
		 The converse part follows directly from Theorem \ref{Unitary}.
	\end{proof}
\begin{cor}\label{Abelian_Unitary} Let $U\colon G\to \mathcal{B}(\mathcal{H})$ be a
	unitary representation. Then $U $ is a commutative unitary representation if and only if
	\[
	U|_{[G,G]}=I_\HS,
	\]
	where, $[G, G]=\langle sts^{-1}t^{-1}\;\colon s,t\in G\rangle$ is
	the commutator subgroup of the group $G.$
	
	\begin{proof}
		 This is an easy consequence of the fact that $U$ is a group homomorphism.
		\begin{equation*}
			\begin{split}
				U(s)U(t)=U(t)U(s)  \iff U(st)=U(ts)
				 \iff U(ts)^{-1}U(st)=I_\HS
				\iff U(s^{-1}t^{-1}st)=I_\HS,
			\end{split}		
		\end{equation*}
for every $s,t \in G.$	
	\end{proof}
\end{cor}

\begin{eg}[Unitary representation on cyclic groups]\label{UZ_n} Let $G$ be a cyclic group generated by $g_0$ and let  $U: G \to \mathcal{B}(\mathcal{H})$ be a  unitary representation. Then $U(g)=U(g_0^m)=U(g_0)^m$ for every $g \in G$ and some $m \in \mathbb{N}$. Thus, $U$ is completely determined by the choice of $U(g_0).$ 
	\qed
\end{eg}


\begin{eg}[Commutative unitary representation on symmetric groups]\label{US_n} Let $S_n$ be the symmetric group on $n$ elements for $n \geq 3$. Let $U\colon S_n\to\mathcal{B}(\mathcal{H})$ be a commutative unitary representation.
	The commutator subgroup $[S_n,S_n]$ is the normal subgroup $A_n$
	consisting of all even permutaions and $S_n/A_n\cong\mathbb{Z}_2.$ Since $U$ is
	commutative, Corollary \ref{Abelian_Unitary} yields that $U(\sigma)=I_\HS$ for all
	$\sigma\in A_n.$ Consequently, there is a well-defined unitary representation $\widetilde{U}\colon
	S_n/A_n\to\mathcal{B}(\mathcal{H}), \sigma A_n \mapsto U(\sigma)$ such that the following diagram commutes.
	\[
	\xymatrix
	{ 
		S_n\ar[r]^{U}\ar[d]_{\pi} & \mathcal{B}(\mathcal{H})\\
		S_n/A_n\ar[ur]_{\widetilde{U} }
	}
	\] 
Any unitary representation on a group of order $2$ is completely determined by a self-adjoint unitary. Thus, $\widetilde{U}$ is uniquely determined by a self-adjoint unitary operator, say, $U_0$. Hence
\[
 U(\sigma)=\begin{cases}I_\HS & {\text{ if }}
		~\sigma\in ~A_n \\ U_0 & {\text{ if }} ~ \sigma \notin ~A_n .
		\end{cases} 
\] 
\qed
\end{eg}

\begin{eg}[Commutative unitary representation on the dihedral groups]\label{UD_n}
	Let $D_n$ be the dihedral group of order $2n$ which has the following representation.
	\[
	D_n=\langle r,s \ | \ r^n=s^2=(sr)^2=e \rangle.
	\]
The commutator subgroup of $D_n$ is the cyclic subgroup
	generated by $r^2 $ and we have that
	\[
	 [D_n,D_n]=\begin{cases}
		\langle r^2\rangle=\langle r\rangle & {\text{ if $n$ is odd}}\\
		\langle r^2\rangle & {\text{ if $n$ is even . }} \end{cases}
	\]
	Let $U\colon D_{n}\to\mathcal{B}(\mathcal{H})$ be a commutative unitary representation. We consider two different cases depending on whether $n$ is odd or even.
\begin{enumerate}
	\item Let $n$ be odd. Then $D_n/[ D_n,D_n ]\cong\mathbb{Z}_2.$ Since $U$ is commutative,
	Corollary \ref{Abelian_Unitary} implies that $U(r^j)=I_\HS$ for  $j=1, \dotsc, n.$ Consequently, there is a well-defined unitary representation $\widetilde{U}\colon
	D_n/[D_n,D_n]\to\mathcal{B}(\mathcal{H})$ such that the following diagram commutes.
	\[
	\xymatrix
	{ 
		D_n\ar[r]^{U}\ar[d]_{\pi} & \mathcal{B}(\mathcal{H})\\
		D_n/[D_n,D_n]\ar[ur]_{\widetilde{U} } 
	}
	\] 
	Indeed, $\widetilde{U}(x [D_n,D_n]) =U(x)$ for all $x\in D_n.$ Since $\widetilde{U}$ is uniquely determined by a self adjoint unitary operator $U_0$, we have that
	\[
	U(x)=\begin{cases}
		I_\HS & {\text{ if }} ~x\in ~\langle r \rangle \\
		U_0 & {\text{ if }} ~x\notin ~\langle r \rangle.
	\end{cases}
\]
\item 	Let $n$ be even. The commutativity of $\{U(x): x \in D_n\}$ and Corollary \ref{Abelian_Unitary} implies that $U(x)=I_\HS$ for all $x\in [D_n, D_n]=\langle r^2 \rangle.$ Since $D_n$ is generated by two element $r$ and $s,$ it is enough to study the
images of the generators $r$ and $s.$ We denote $U(r)=U_r$ and $U(s)=U_s$ for which we have that $(U_r)^n=U(r^n)=I_\HS$ and $(U(s))^2=U(s^2)=I_\HS.$ Consequently, we get that
 \[
 U(x)=\begin{cases}
	I_\HS & {\text{ if }} ~x \in \langle r^2 \rangle  \\ 
	U_r & {\text{ if }} ~x=r,r^3,r^5,\dotsc,r^{2n-1} \\
	U_s & {\text{ if }} ~x=r^{j}s \mbox{ and $j$ is even}\\
	U_rU_s & {\text {if }} ~x=r^{j}s\mbox{ and $j$ is odd},
\end{cases}
\]
where $U_r$ and $U_s$ are commuting self-adjoint unitaries.
\end{enumerate}	
\qed
\end{eg}


\section{Powers of positive definite functions and unitary representations}\label{power dilation}

\vspace{0.3cm}		
	
\noindent Given a map $T: G \to \mathcal{B}(\mathcal{H})$ and $n \in \mathbb{N},$ one can define another map $T_{n}: G \to \mathcal{B}(\mathcal{H})$ by $T_n(s)= T(s)^n$. There are a few interesting questions that we address for the map $T_n$.
\begin{enumerate}
\item If $T$ is a positive definite function, then is $T_n$ positive definite too?
\item  Can we say that $T_n$ is a unitary representation if $T$ is so?	
	\end{enumerate} 
In this Section, we show that the map $T_n$ need not inherit these properties. We further identify some classes of operator-valued functions for which the above questions have an affirmative answer. If $T$ is an operator-valued function on a group $G$ such that $T(s^{-1})=T(s)^*$ for every $s \in G,$ then we have that
	\[
	T_n(s^{-1})=T(s^{-1})^n=T(s)^{*n}=T_n(s)^*
	\]
	for every $s \in G.$ Consequently, for a positive definite function $T$ on a group $G$, the operator-valued function is positive definite if and only if 
	\[
	\underset{s,t \in G}{\sum}\langle T_n(s^{-1}t)h(t),h(s)\rangle \geq 0, 
	\]
for every $h \in c_{00}(G, \mathcal{H}).$ For a group $G$, Proposition \ref{+ arbitrary} implies that $T_n$ is a positive definite function on $G$ if and only if the associated block matrix 
\begin{equation}\label{Hadamard}
	\begin{split} 
		\Delta_{T_n}(s_1, \dotsc, s_m)
&=\begin{bmatrix}
		T(s_i^{-1}s_j)^n
	\end{bmatrix}_{1 \leq i,j \leq m}\\
\end{split} 
\end{equation}
is a positive definite for every $\{s_1, \dotsc, s_m\}$ in $G$. In this connection, let us define \textit{block Hadamard product} of block matrices.
\begin{defn}
Let $A=[A_{ij}]$ and $B=[B_{ij}]$ be $p \times p$ block matrices in which each block is an operator. The \textit{block Hadamard product} $A\square B$ is defined as $A \square B:=[A_{ij}B_{ij}],$ where $A_{ij}B_{ij}$ denotes the usual composition of the operators $A_{ij}$ and $B_{ij}.$
\end{defn}
\begin{lem}\label{Hadamard arbitrary}
	Given a positive definite function $T$ on a group $G,$ the operator-valued function $T_n$ is positive definite if and only if 
	\[
\underbrace{\Delta_T(s_1, \dotsc, s_m) \square \dotsc \square  \Delta_T(s_1, \dotsc, s_m),
}_{n-times}
		\]
 is positive for every $\{s_1, \dotsc, s_m\} \subseteq G.$
\end{lem}
\begin{proof}
The block matrix in (\ref{Hadamard}) can be re-written as 
$$
		\Delta_{T_n}(s_1, \dotsc, s_m)= 	\underbrace{\Delta_T(s_1, \dotsc, s_m) \square  \dotsc \square  \Delta_T(s_1, \dotsc, s_m)
		}_{n-times}.
$$
The desired conclusion now follows directly from the above discussion.
\end{proof}

The next corollary is a direct consequence of Proposition \ref{+ finite} and Lemma \ref{Hadamard arbitrary}.
\begin{cor}\label{Hadamard finite}
	Given a positive definite function $T$ on a finite group $G,$ the map $T_n$ is positive definite if and only if 
	$
	\underbrace{\Delta_T(G) \square \dotsc \square  \Delta_T(G)
	}_{n-times}
	$
	is positive.
\end{cor}

 The following example shows that $T_n$ need not be a positive definite function for any $n\geq 2$ even if $T$ is a positive definite function. 
\begin{eg}
	Consider the operator-valued function $T: \mathbb{Z}_2 \to \mathcal{B}(\mathbb{C}^2)$ given by 
	\[
	T(0)=\begin{bmatrix}
		2 & 0\\
		0 & 1
	\end{bmatrix}, \quad T(1)=\begin{bmatrix}
		-1 & -1\\
		-1 & 0
	\end{bmatrix}.
	\]
	Since $T(0)$ and $T(1)$ are self-adjoint, $T(s^{-1})=T(s)^*$ for every $s \in \mathbb{Z}_2.$ Consequently, Proposition \ref{+ finite} yields that $T$ is positive definite if and only if the block matrix
	\[
	\Delta_T=\begin{bmatrix}
		T(0) & T(1)\\
		T(1) & T(0)\\
	\end{bmatrix}
	=\begin{bmatrix}
		2 & 0 & -1 & -1\\
		0 & 1 & -1 & 0 \\
		-1 & -1 & 2 & 0\\
		-1 & 0 & 0 & 1\\
	\end{bmatrix}
	\]
	is positive. Since $\Delta_T$ is a self-adjoint matrix with eigenvalues $\{0, 2, 2 \pm \sqrt{2}\},$ we have that $\Delta_T$ is positive. Corollary \ref{Hadamard finite} implies that the map $T_n: \mathbb{Z}_2 \to \mathcal{B}(\mathcal{H}), s \mapsto T(s)^n$ if and only if the Hadamard product of $\Delta_T$ with itself $n$-times is positive. We have that 
	\[
	\Delta_{T_2}=\begin{bmatrix}
		T(0)^2 & T(1)^2\\
		T(1)^2 & T(0)^2\\
	\end{bmatrix}
	=\begin{bmatrix}
		4 & 0 & 2 & 1\\
		0 & 1 & 1 & 1 \\
		2 & 1 & 4 & 0\\
		1 & 1 & 0 & 1\\
	\end{bmatrix} \quad \mbox{and} \quad 	\Delta_{T_3}=\begin{bmatrix}
	T(0)^3 & T(1)^3\\
	T(1)^3 & T(0)^3\\
\end{bmatrix}
=\begin{bmatrix}
8 & 0 & -3 & -2\\
0 & 1 & -2 & -1 \\
-3 & -2 & 8 & 0\\
-2 & -1 & 0 & 1\\
\end{bmatrix}
	\]
	are self-adjoint matrices with $\det(\Delta_{T_2})=-11$ and $\det(\Delta_{T_3})=-288.$ Conseqeuntly, $T_2$ and $T_3$ are not positive definite functions. Infact, we show that there is no $n \geq 2$ for which $T_n$ is a positive definite function. We first compute $T(1)^n$ for which we define a recursive sequence
	\[
	a_1=1, a_2=2, a_n=a_{n-1}+a_{n-2} \quad n\geq 3,
	\]
and using induction argument, we have that 
\[
T(1)^n=(-1)^n\begin{bmatrix}
	a_n & a_{n-1}\\
	a_{n-1} & a_{n-2}\\ 
\end{bmatrix}.
\]	
For $n \geq 3,$ we have the following.
\[
\Delta_{T_n}=\begin{bmatrix}
	T(0)^n & T(1)^n\\
	T(1)^n & T(0)^n\\
\end{bmatrix}
=\begin{bmatrix}
	2^n & 0 & (-1)^na_n & (-1)^na_{n-1}\\
	0 & 1 & (-1)^na_{n-1} & (-1)^na_{n-2} \\
	(-1)^na_{n} & (-1)^na_{n-1} & 2^n & 0\\
	-(-1)^na_{n-1} & (-1)^na_{n-2} & 0 & 1\\
\end{bmatrix},
\]
whose determinant is given by
\[
\det(\Delta_{T_n})=(4^n-a_n^2)(1-a_{n-2}^2)-a_{n-1}^2(a_{n-1}^2+2a_na_{n-2}).
\]
For any $n \geq 3, a_{n-2} \geq 1$ and inductively, it follows that $a_n<2^n$ which yields that $(4^n-a_n^2)(1-a_{n-2}^2) \leq 0.$ Since $a_{n-1}^2(a_{n-1}^2+2a_na_{n-2})$ is strictly positive, we have that $\det(\Delta_{T_n})$ is strictly negative.
\qed
\end{eg}
Since a unitary representation on a group is a homomorphism, the image of an abelian group under a unitary representation consists of commuting unitaries. However the above example shows that the image of an abelian group under a positive definite function need not be consisting of commuting operators. Next, we study a class of positive definite functions $T$ for which $T_n$ is always positive definite for any $n \in \mathbb{N}$. We begin with the following result from the literature.
\begin{thm}[\cite{Gunther}, Corollary 3.3]\label{positive Hadamard}
Let $A=[A_{ij}]$ and $B=[B_{ij}]$ be $p \times p$ block matrices in which each block is an $n \times n$ matrix with complex enteries. If $A$ and $B$ are (strictly) positive definite such that every block of $A$ commutes with every block of $B$, then $A\square B$ is (strictly) positive definite.
\end{thm}
\begin{thm}
	If $T: G \to \mathcal{B}(\mathbb{C}^k)$ is a positive definite function on a group $G$ such that $\{T(s): s \in G\}$ is a family of commuting $k \times k$ matrices, then $T_n$ is a positive definite function for every $n \in \mathbb{N}.$
\end{thm}
\begin{proof}
	It follows from Lemma \ref{Hadamard arbitrary} that $T_n$ is positive definite if and only if $\Delta_{T_n}(s_1, \dotsc, s_m)$ (which is $n$-times Hadamard product of $\Delta_T(s_1, \dotsc, s_m)$ with itself) is positive for every $\{s_1, \dotsc, s_m\} \subseteq G$. By Proposition \ref{+ arbitrary},
	$
		\Delta_T(s_1, \dotsc, s_m)=\begin{bmatrix}
T(s_{i}^{-1}s_j)
	\end{bmatrix}_{i, j=1}^{m}
	$
is positive. Since $\{T(s): s \in G\}$ is a family of commuting matrices, Theorem \ref{positive Hadamard} implies that $\Delta_T(s_1, \dotsc, s_m) \square \Delta_T(s_1, \dotsc, s_m)$ is positive for every $\{s_1, \dotsc, s_m\} \subseteq G$. Consequently, $T_2$ is positive definite. The desired conclusion follows from mathematical induction.
\end{proof}

The above result can be extended to commuting normal operators on a Hilbert space. 
\begin{thm}[\cite{Krishna}, Theorem 2.3]\label{Normal Hadamard} Let $\mathcal{A}$ be a commutative unital $C^*$-algebra. If $M, N \in M_n(\mathcal{A})$ are positive, then their Hadarmard product $M\square N$ is positive.
\end{thm}
\begin{thm}
	If $N: G \to \mathcal{B}(\mathcal{H})$ is a positive definite function acting on a group $G$ such that $\{N(s): s \in G\}$ is a family of commuting normal operators, then $N_k$ is a positive definite function for every $k \in \mathbb{N}.$
\end{thm}
\begin{proof}
	Let $\mathcal{A}$ be the unital  $*$-algebra generated by $\{I_\mathcal{H}, N(s): s \in G\}$. Since $\{N(s): s\in G\}$ is a family of commuting normal operators, $\mathcal{A}$ is a commutative unital $C^*$-algebra. By Lemma \ref{Hadamard arbitrary}, $N_k$ is positive definite if and only if $\Delta_N (s_1, \dotsc, s_m)\square \dotsc \square \Delta_N (s_1, \dotsc, s_m)$ ( $k$-times) is positive for every $\{s_1, \dotsc, s_m\} \subseteq G.$ Since
$ \Delta_N(s_1, \dotsc, s_m)=\begin{bmatrix}
		N(s_{i}^{-1}s_j)
	\end{bmatrix}_{i, j=1}^ m
	$
	is positive and each entry in $\Delta_{N_k}$ is in $\mathcal{A}$, it follows from Theorem \ref{Normal Hadamard} that $\Delta_N(s_1, \dotsc, s_m) \square \Delta_N(s_1, \dotsc, s_m)$ is positive for every $\{s_1, \dotsc, s_m\} \subseteq G$. Thus, $N_2$ is positive definite. The desired conclusion follows from induction.
\end{proof}

 We now characterize unitary representations $U$ for which $U_n ( n \in \mathbb N)$ is a unitary representation too. 
\begin{lem}\label{Unitary_Power-Dilation}
	Let $U\colon G\to \mathcal{B}(\mathcal{H})$ be a unitary representation acting on a group $G$ and let $n \in \mathbb{N}.$ Then $U_{n}\colon G\to \mathcal{B}(\mathcal{H}) , s \mapsto U(s)^n$ is a unitary representation if and only if $U|_{G_n}=I_\HS$ where, $G_n=\langle s^{-n+1} (ts)^{n-1} t^{-n+1} \ | \ s,t\in G\rangle.$
\end{lem}
	\begin{proof}
		Given a unitary representation $U$ on a group $G$ and $n \in \mathbb{N},$ we have that $U_n(s)$ is unitary for every $s \in G$ and $U_n(e)=I_\HS.$ Therefore, $U_n$ is unitary representation if and only if $U_n$ is a group homomorphism. Since $U$  is a homormophism, we have that 
		\begin{equation*}
			\begin{split}
			U_n(st)=U_n(s)U_n(t) & \iff U(st)^n=U(s)^nU(t)^n\\
			& \iff \left( U(s)U(t)\right)^n=U(s)^nU(t)^n\\
			& \iff 	\left( U(t)U(s)\right)^{n-1}=U(s)^{n-1}U(t)^{n-1}\\
& \iff 	 U(ts)^{n-1}=U(s)^{n-1}U(t)^{n-1}\\
& \iff U(s^{-n+1}(ts)^{n-1}t^{-n+1})=I_\HS,
			\end{split}
		\end{equation*}
for every $s, t \in G.$  The rest follows from the fact that $U$ is a group homomorphism.
	\end{proof}
\begin{cor}\label{Abelian_UnitaryII} Let $U\colon G\to \mathcal{B}(\mathcal{H})$ be a
	unitary reprsentation. Then $U_2 $ is a unitary representation if and only if $\{U(s) : s\in G \}$ is a family of commuting operators.
	\begin{proof}
	By Lemma \ref{Unitary_Power-Dilation}, $U_2$ is a unitary representation if and only $U(s^{-1}tst^{-1})=I_\HS$ for every $s, t \in G$. Since $U$ is a unitary representation, the latter condition holds if and only $U(s)U(t)=U(t)U(s)$ for every $s, t \in G$ . 
	\end{proof}
\end{cor}

Next, we show that $U_n$ need not be a unitary representation even if $U$ is a unitary representation.
\begin{eg}\label{S_3 unitary}
	Let $\{e_1, e_2, e_3\}$ be the standard orthonormal basis of $\mathbb{C}^3$ over $\mathbb{C}$ and let $S_3$ be the symmetry group on a set of three elements. Consider the permutation representation 
	\[
	U:S_3 \to \mathcal{B}(\mathbb{C}^3) \quad \text{defined as} \quad  U(\sigma)(e_i)=e_{\sigma(i)},
	\] 
whose image consists of permutation matrices. Consequently, $U$ is a unitary representation but the image does not consist of commuting operators. Corollary \ref{Abelian_UnitaryII} implies that $U_2$ is not a unitary representation on $S_3.$
\qed 
\end{eg}

 The above example also shows that the image of a group under unitary representation need not consist of commuting unitaries. In the next example, we show that even if the image of a positive definite function $T$ on a group $G$ consists of commuting operators, there is a unitary representation $U$ on $G$ which does not consist of commuting operators but $T(s)=P_\HS U(s)|_\HS$ for every $s \in G.$
\begin{eg}
On the group $S_3=\{(1), (12), (13), (23), (123), (132)\},$
define an operator-valued function $T: S_3 \to \mathbb{C}$ as
\[
T(\sigma)=\begin{cases}
	1 & {\text{ if }} \sigma \ \mbox{fixes} \ 1 \\
	0 &  \mbox{otherwise}.
\end{cases}
\]	
It is easy to see that $T(\sigma^{-1})=T(\sigma)^*$ for every $\sigma \in S_3.$ For any $h \in c_{00}(S_3, \mathbb{C})$, we have 
\[
\underset{\sigma, \tau \in S_3}{\sum}\langle T(\sigma^{-1}\tau)h(\tau), h(\sigma)\rangle=|z_1|^2+2Rez_1\overline{z}_2+|z_2|^2=|z_1+z_2|^2\geq 0,  
\]	
where $z_1$ and $z_2$ are the images of the elements $(1)$ and $(23)$ under $h$ respectively. Consequently, $T$ is a positive definite function.	Define an isometry $V: \mathbb{C} \to \mathbb{C}^3$ by $Vz=(z, 0, 0)$. Let $U$ be the permutation unitary representation on $S_3$ as seen in Example \ref{S_3 unitary} for which we have that $T(s)=V^*U(s)V$ for every $s \in S_3$ but the image of $U$ does not consist of commuting operators.
\qed 
\end{eg}

\begin{rem}\label{U_nT_n}
	Let $T\colon G\to \mathcal{B}(\mathcal{H})$ be a positive definite function with
	$T(e)=I_\HS$ and $U\colon G\to \mathcal{B}(\mathcal{K})$ be a unitary representation such that $T(s)= P_\mathcal{H} U(s)|_\mathcal{H}$ for every $s \in G.$ In general, $T_n$ need not be a positive definite function and $U_n$ need not be a unitary representation for a given $n \in \mathbb{N}$ but if assume that such phenomenon occurs for some $n \in \mathbb{N},$ then a natural question is if $T_n(s)= P_\mathcal{H} U_n(s)|_\mathcal{H}$ for every $s \in G.$ The following example shows that this is not true in general.
\end{rem}
\begin{eg}
	Consider the operator-valued function $T:\mathbb{Z}_2 \to \mathcal{B}(\mathbb{C}^2)$ given by 
	\[
	T(0)=\begin{bmatrix}
		1 & 0\\
		0 &1\\
	\end{bmatrix}, \quad 
T(1)=\begin{bmatrix}
	0 & 1\slash 2\\
	1\slash 2 & 0 \\
\end{bmatrix}.
	\]
It is easy to see that $T(s^{-1})=T(s)^*$ for every $s \in \mathbb{Z}_2.$ Since $T(1)$ is a self-adjoint contraction, Corollary \ref{+ on Z_2II} implies that $T$ is positive definite. Similarly, $T_2$ is positive definite too. It follows from Theorem \ref{NeumarkII} that there is a unitary representation $U$ on $\mathbb{Z}_2$ such that $T(s)= P_\mathcal{H} U(s)|_\mathcal{H}$ for every $s \in \mathbb{Z}_2$. Since the image of $U$ consists of two unitaries and one of them is identity, therefore, $U$ consists of commuting unitaries. Corollary \ref{Abelian_UnitaryII} implies that $U_2$ is also a unitary representation on $\mathbb{Z}_2.$ Let if possible, $T(s)= P_\mathcal{H} U(s)|_\mathcal{H}$ for every $s \in \mathbb{Z}_2.$ Then
\begin{equation*}
	\begin{split}
		T(1)^2=T_2(1)=P_\mathcal{H}U_2(1)|_\mathcal{H}=P_\mathcal{H}U(1)^2|_\mathcal{H}=P_\mathcal{H}U(0)|_ \mathcal{H}=T(0),
	\end{split}
\end{equation*}
which is a contradiction. \qed
\end{eg}

We conlcude this article by providing a necessary and sufficient condition for which $T_n(s)=P_\HS U_n(s)|_\HS$ under the hypothesis given in Remark \ref{U_nT_n}.
\begin{prop} Let $T\colon G\to \mathcal{B}(\mathcal{H})$ be a positive definite function with
	$T(e)=I$ and $U\colon G\to \mathcal{B}(\mathcal{K})$ be a unitary representation such that $T(s)= P_\mathcal{H} U(s)|_\mathcal{H}$ for every $s \in G.$ Assume that for some  $n \in \mathbb{N}, T_n$ is a positive definite function and $U_n$ is a unitary representation. Then $T_n(s)= P_\mathcal{H} U_n(s)|_\mathcal{H}$ for every $s \in G$ if
	and only if $T(s^n)=T(s)^n$ for every $s\in G.$
\end{prop}
	\begin{proof}
		Let $s \in G$. Since $T(s)= P_\mathcal{H} U(s)|_\mathcal{H}$, we have
		\begin{equation*}
			\begin{split}
				P_{\mathcal{H}}U_{n}(s)\vert_{\mathcal{H}}=T_{n}(s) 
				& \iff    P_{\mathcal{H}}U(s)^{n}\vert_{\mathcal{H}}=T(s)^{n} \\
				& \iff P_{\mathcal{H}}U(s^n)\vert_{\mathcal{H}}=T(s)^{n}\\
				&\iff
				T(s^n)=T(s)^{n}.\\
			\end{split}
		\end{equation*}
		The proof is complete.
	\end{proof}

	\end{document}